\documentclass[reqno]
{amsproc}
\usepackage{amsmath}
\usepackage{amssymb}
\usepackage{graphicx}
\usepackage{multirow}
\DeclareGraphicsExtensions{.png,.pdf,.eps}
\usepackage[all,2cell]{xy}
\usepackage{tikz}
\usepackage{pgf}
\usepackage{enumerate}
\usepackage{mathdots}
\usepackage{enumitem}
\usepackage{hf-tikz}
\usetikzlibrary{tikzmark,quotes}

\usetikzlibrary{cd}
\usepackage{tikz-cd}
\usepackage{hyperref}

\newtheorem{theorem}{Theorem}[section]
\newtheorem{lemma}[theorem]{Lemma}
\newtheorem{corollary}[theorem]{Corollary}
\newtheorem{proposition}[theorem]{Proposition}

\theoremstyle{definition}

\newtheorem{definition}[theorem]{Definition}
\newtheorem{notation}[theorem]{Notation}
\newtheorem{setup}[theorem]{Setup}

\newtheorem{remark}[theorem]{Remark}

\newtheorem{set}[theorem]{Setup}
\newtheorem{example}[theorem]{Example}

\newtheorem{construction}{Construction}

\def\cMU{{\mathcal U}}
\usepackage{array}
\newcolumntype{L}{>{$}l<{$}} 
\newcolumntype{C}{>{$}c<{$}} 

\title[Characterizing rational homogeneous spaces via $\C^*$-actions]{Characterizing rational homogeneous spaces via $\C^*$-actions}

\author[Occhetta]{Gianluca Occhetta}
\address{Dipartimento di Matematica, Universit\`a degli Studi di Trento, via
Sommarive 14 I-38123 Povo di Trento (TN), Italy}
 
\email{gianluca.occhetta@unitn.it, eduardo.solaconde@unitn.it}

\author[Sol\'a Conde]{Luis E. Sol\'a Conde}


\thanks{Authors supported by the Department of Mathematics of the University of Trento, and partially supported by INdAM--GNSAGA. The authors would like to thank Jarek Wi\'sniewski for inspirational conversations and useful remarks, and  an anonymous referee for his
comments and suggestions, which helped improving the exposition of the paper.}

\usepackage[textsize=tiny]{todonotes}

\usepackage{enumitem}

\newcommand\ignore[1]{}

\DeclareMathOperator{\HH}{H}

\newcommand\PP{{\mathbb{P}}}

\newcommand\QQ{{\mathbb{Q}}}

\def\C{{\mathbb C}}

\def\P{{\mathbb P}}
\def\Q{{\mathbb Q}}

\def\Z{{\mathbb Z}}

\def\cD{{\mathcal D}}
\def\cE{{\mathcal E}}

\def\cG{{\mathcal{G}}}
\def\cH{{\mathcal{H}}}

\def\cL{{\mathcal L}}
\def\cM{{\mathcal M}}
\def\cN{{\mathcal{N}}}
\def\cO{{\mathcal{O}}}

\def\cU{{\mathcal U}}

\def\cX{{\mathcal X}}
\def\cY{{\mathcal Y}}

\def\Q{{\mathbb{Q}}}

\def\fg{{\mathfrak g}}
\def\fh{{\mathfrak h}}
\def\fp{{\mathfrak p}}
\def\fb{{\mathfrak b}}

\def\operatorname#1{\mathop{\rm #1}\nolimits}

\def\DA{{\rm A}}
\def\DB{{\rm B}}
\def\DC{{\rm C}}
\def\DD{{\rm D}}
\def\DE{{\rm E}}

\def\Proj{\operatorname{Proj}}
\def\Aut{\operatorname{Aut}}

\def\Chow{\operatorname{Chow}}

\def\Hom{\operatorname{Hom}}

\def\Pic{\operatorname{Pic}}
\def\Hom{\operatorname{Hom}}

\def\codim{\operatorname{codim}}

\def\rank{\operatorname{rank}}

\def\det{\operatorname{det}}

\def\loc{\operatorname{Locus}}
\def\rat{\operatorname{RatCurves}}

\def\NE{{\operatorname{NE}}}

\def\Nu{{\operatorname{N_1}}}

\def\GX{\mathcal{G}\!X}
\def\CX{\mathcal{C}\!X}

\newcommand{\pb}{\ar@{}[dr]|{\text{\pigpenfont J}}}
\def\ol{\overline}
\def\tl{\widetilde}

\makeatletter
\newcommand{\xleftrightarrow}[2][]{\ext@arrow 3359\leftrightarrowfill@{#1}{#2}}
\makeatother
\newcommand{\xdasharrow}[2][->]{
\tikz[baseline=-\the\dimexpr\fontdimen22\textfont2\relax]{
\node[anchor=south,font=\scriptsize, inner ysep=1.5pt,outer xsep=2.2pt](x){#2};
\draw[shorten <=3.4pt,shorten >=3.4pt,dashed,#1](x.south west)--(x.south east);
}}

\newcommand\m{{\mathfrak m}}

\newcommand\lra{\longrightarrow}

\def\Mo{\operatorname{\hspace{0cm}M}}

\def\prul{s}
\def\prur{d}
\def\orbits{\prec}
\def\orbit{\preceq}
\newcommand{\lb}{D}
\begin{document}
\begin{abstract}
We study smooth varieties of Picard number one admitting a special dominating family of rational curves and an equalized $\C^*$-action. In particular we show that $X$ is a smooth variety of Picard number one with nef tangent bundle admitting an equalized $\C^*$-action with an isolated extremal fixed point if and only if $X$ is an irreducible Hermitian symmetric space.
\end{abstract}

\maketitle

\tableofcontents

\section{Introduction}

One long standing open problem in the framework of Fano manifolds is the  Campa\-na--Peternell (CP, for short) Conjecture, that states that the only Fano manifolds with nef tangent bundles are rational homogeneous varieties. In other words, for a Fano manifold $X$ the nefness assumption on $T_X$ should imply the existence of an affine algebraic group acting transitively on $X$. One usually considers the case in which the Picard number $\rho_X$ is equal to one, from which the general case should follow. For more information on the Conjecture and further references we refer the reader to the survey \cite{MOSWW}.

In view of Mori's proof of the Hartshorne--Fr\"ankel Conjecture -- of which the CP Conjecture is a natural generalization -- 
one of the  strategies that has been used to prove partial results is the study of  rational curves. If the tangent bundle $T_X$ is nef, then the deformations of rational curves in $X$ are unobstructed, and one may try to reconstruct the geometry of $X$ out of certain families of rational curves. Of particular interest in this sense are the so-called unsplit and dominating families of rational curves, that allow us to consider some sort of infinitesimal-to-global arguments. More concretely, many results in the literature show that the homogeneity of a Fano manifold $X$ of Picard number one can be inferred from the homogeneity of the sets of tangent directions to the curves of an unsplit and dominating family of rational curves at points (see, for instance \cite{Mk3,HH,OSWi}). 

A natural question related to this approach is  to compare the nefness hypothesis with other (a priori weaker) assumptions on a Fano manifold $X$, such as:
\begin{itemize}
\item[(C)] {\em convexity}: $X$ is called convex if the restriction of $T_X$ to the normalization of every rational curve in $X$ is nef;
\item[(B)] {\em beauty}: $X$ is called beautiful if it contains a beautiful family of rational curves, i.e., an unsplit covering family of rational curves such that $T_X$ is nef on every curve of the family.  
\end{itemize}

Clearly the nefness of the tangent bundle implies convexity, which implies the existence of a beautiful family of rational curves. To our best knowledge, the converses of these implications -- in the case of Fano manifolds of Picard number one -- are not known, so we may ask ourselves whether these properties characterize rational homogeneous varieties.

In this paper we will work under the hypothesis (B), and we will assume that the group of automorphisms of $X$ contains a $1$-dimensional torus satisfying certain assumptions; the statement of our main result is the following:

\begin{theorem}\label{thm:CP1}
Let $X$ be a smooth complex projective variety of Picard number one satisfying {\rm (B)}. Assume that it admits an equalized $\C^*$-action with an isolated extremal fixed point. Then $X$ is an irreducible Hermitian symmetric space. 
\end{theorem} 

The equalization assumption means that the $\C^*$-action does not have finite isotropy subgroups; this is a technical condition that implies some desirable properties for the family of $\C^*$-invariant rational cycles in $X$ containing the closure of the general orbit (see \cite{WORS6}). Every rational homogeneous variety associated to a Lie algebra of of type $\DA,\DB,\DC,\DD,\DE_6,\DE_7$ admits an equalized $\C^*$-action, but only  the irreducible Hermitian symmetric spaces (see Table \ref{tab:ihs} for their list) admit an equalized action with an isolated extremal fixed point. 

Another characterization of (some) irreducible Hermitian symmetric spaces in terms of $\C^*$-actions on smooth projective varieties of Picard number one has been recently presented by Liu in \cite{Liu23}. His statement does not contain hypotheses on families of rational curves on $X$, but assumes the existence of a sufficiently large group of automorphisms. More precisely he assumes  that  given two general points $x,y \in X$, there exists a $\C^*$-action whose sink and source are $x$ and $y$, and that the action is equalized at these points. 

In order to prove Theorem \ref{thm:CP1} we first study, in Section \ref{sec:wfpc}, the induced $\C^*$-actions on the beautiful family $\cM$ of rational curves in $X$ and on the universal family $\cU$, without any assumptions on the extremal fixed point components.
We prove that both the induced actions on $\cM$ and $\cU$ are equalized and we describe
 the induced action on the variety $\cU_y$, defined as the inverse image of $y\in X$ by the evaluation morphism $q:\cU \to X$. 
 
We then show, in Section \ref{sec:fpc}, that the action on $X$ is quite special: polarizing $X$ with the ample generator $L$ of its Picard group, the weights of the action on $(X,L)$ are (up to normalization) consecutive multiples of the  $L$-degree of curves in $\cM$, there exists only one fixed point component of the action for each weight and consecutive components are joined by orbits whose closures belong to $\cM$ (Theorem \ref{thm:actionx}). 
Moreover, we show that the Bia{\l}ynicki-Birula decompositions of $X$ are stratifications (Proposition \ref{prop:BB1}) and the orbit graph of the action on $X$ is complete  (Proposition \ref{prop:BB2}).

In Section \ref{sec:isolated} we first refine the above results in the case in which the action has an isolated extremal fixed point. In particular, if $y$ is a fixed point belonging to  the fixed point component joined to the isolated point by curves in $\cM$, the action on $\cU_y$ has an extremal fixed point, too. Since the criticality of the action on $\cU_y$ is low, these varieties can be classified: Section \ref{sec:low} is devoted to this task (see Theorems \ref{thm:rho1} and \ref{thm:CP1rho2}).

The proof of Theorem \ref{thm:CP1} then proceeds by discarding the non-homogeneous possibilities for $\cU_y$ and by showing that, for every $x \in X$ the variety $\cU_x$ is isomorphic to $\cU_y$.
The homogeneity of $X$  then follows from  \cite[Theorem~1.1]{OSWi}, and the fact that $X$ is an irreducible Hermitian symmetric space from the classification of equalized $\C^*$-actions on homogeneous varieties (see \cite[Table 2]{WORS5}).

\section{Preliminaries}

\subsection{$\C^*$-actions}\label{ssec:c*}

In this section we recall some background material on $\C^*$-actions. We refer to \cite{BB,CARRELL,WORS1} for further details. 
Unless otherwise stated, $X$ will be a smooth projective variety, endowed with a non-trivial $\C^*$-action. Given a point $x\in X$, we will define:
$$
x_-:=\lim_{t\to 0}t^{-1}x, \qquad x_+:=\lim_{t\to 0}tx,
$$
and call $x_-,x_+$ the {\em sink} and the {\em source} of the orbit $\C^*x$. 
We denote by $X^{\C^*}$  the fixed locus of the action, and by $\cY$ the set of irreducible fixed point components: 
$$X^{\C^*}=\bigsqcup_{Y\in \cY}Y.$$ 
The {\em sink} and the {\em source} of the action, often called the {\em extremal fixed point components}, are the fixed point components $Y_-,Y_+\subset X^{\C^*}$ containing, respectively, the limit points  $x_-$ and $x_+$ of a general  point $x\in X$. 
The fixed point components different from the sink and the source are called {\em inner}.
Every $Y\in\cY$ is smooth (cf. \cite{IVERSEN}),
and the $\C^*$-action on $TX_{\mid Y}$ gives a  decomposition
\begin{equation} \label{eq:dectang}
(T_X)_{\mid Y} =T^{+}\oplus T^{0}\oplus T^{-},
\end{equation}
where $T^{+}$, $T^{0}$, $T^{-}$ are  the
subbundles of $(T_X)_{\mid Y}$ where $\C^*$ acts with positive,
zero or negative weights, respectively. Then, by local linearization, $T^0=T_Y$ and
\begin{equation} \label{eq:decnorm}
T^{+}\oplus T^{-}=N_{Y/X}=N^+(Y)\oplus N^-(Y)
\end{equation}
is the decomposition of the normal bundle $N_{Y/X}$ into summands on which $\C^*$ acts with positive and negative weights, respectively.
We set
\begin{equation} \label{eq:nudef}
\nu^{+}(Y):=\rank{N^+(Y)}, \qquad  \nu^{-}(Y):=\rank{N^-(Y)}.
\end{equation}

Throughout the paper we will freely use the notation $Y_{\pm}$, $N^\pm(Y)$, \dots, to refer to $Y_+$ and $Y_-$, $N^+(Y)$ and $N^-(Y)$, \dots

A straightforward consequence of the  decomposition (\ref{eq:decnorm}) is that the tangent bundle of a fixed point component is a quotient of the restriction of the tangent bundle of $X$. 

We can define a partial order in $\cY$ as follows:
given two fixed point components $Y,Y' \in \cY$ we will write 
$$Y \orbit Y'$$
if there exist points 
$x^1, \dots, x^k\in X$ such that $x^1_- \in Y$, $x^k_+ \in Y'$, and,  for every $i=1, \dots, k-1$, the points $x^i_+,x^{i+1}_-$ belong to the same fixed point component.

\subsection{Bia{\l}ynicki-Birula decomposition}

We refer to \cite{CARRELL} for a complete account on the Bia{\l}ynicki-Birula decomposition and its applications, and to \cite{BB} for the original reference. 
Given any set $S \subset X$, we define:
\begin{equation}\label{eq:BBcells}
X^\pm(S):=\{x\in X|\,\, x_\pm\in S\}, \quad
B^\pm(S):=\overline{X^\pm(S)}.
\end{equation}
When $S=Y$ with $Y\in \cY$ we call $X^\pm(Y)$ the positive and the negative {\it Bia{\l}ynicki-Birula cells} of the action at $Y$  (BB-cells, for short), so that we have decompositions: 
$$X=\bigsqcup_{Y\in\cY}X^+(Y)=\bigsqcup_{Y\in\cY}X^-(Y).$$  
We will denote by $B^\pm(Y)$ the closure of the BB-cells $X^\pm(Y)$. Every cell $X^\pm(Y)$ has a natural morphism onto $Y$, sending every $x\in X^\pm(Y)$ to its limiting point $x_\pm$. It was shown by Bia{\l}ynicki-Birula that $X^\pm(Y)$ is an affine bundle, locally isomorphic to $N^\pm(Y)$.

Note that, by definition, $B^\pm(Y_\pm)=X$, that is, the cells $X^\pm(Y_\pm)$ are dense in $X$. In particular $\nu^\mp(Y_\pm)=0$, and $\nu^\pm(Y_\pm)=\dim(X)-\dim(Y_\pm)$.

\begin{remark}\label{rem:strat} It is well known that the Bia{\l}ynicki-Birula ($+$ and $-$) decompositions are not, in general, stratifications: it is not true  that the closure of a BB-cell is a union of BB-cells; a simple example of this kind is \cite[Example~1]{BB76}. 
For a projective variety $X$, the BB-decomposition satisfies a weaker property:  there exists a finite decreasing sequence $X=Z_0\supset Z_1 \supset\dots\supset Z_k=\emptyset$, of closed subschemes such that $Z_i\setminus Z_{i+1}$ is a cell of the decomposition (cf. \cite[Theorem~3]{BB76}).
\end{remark}

\subsection{Actions on polarized pairs}\label{ssec:polar}

Given  a line bundle $L\in \Pic(X)$, there exists a linearization of the $\C^*$-action on it (cf. \cite{KKLV}), so that for every $Y\in \cY$, $\C^*$ acts on $L_{|Y}$ by multiplication with a character $m\in \Mo(\C^*)=\Hom(\C^*,\C^*)$, called {\em weight of the linearization on $Y$}. Fixing an isomorphism $\Mo(\C^*)\simeq \Z$, this defines a map, called {\em weight map}, $\mu_L:\cY\to\Z$.  Two linearizations differ by a character of $\C^*$; in particular, for any  $L \in \Pic(X)$ there exists a (unique) linearization (called {\em normalized linearization}) whose weight at the sink $Y_-$ is equal to zero. 
Unless otherwise stated, $\mu_L$ will  denote the weight map of the normalized linearization of the line bundle $L\in \Pic(X)$. 

\begin{example}\label{ex:antican}
For the anticanonical bundle of $X$, the normalized linearization satisfies (Cf. \cite[Lemma 3.11]{BWW}):
\begin{equation}\label{eq:lincan}
\mu_{-K_X}(Y)= \nu^+(Y)-\nu^-(Y) + \nu^-(Y_-).
\end{equation}
\end{example}
A {\em $\C^*$-action on the polarized pair $(X,L)$}
is a $\C^*$-action on $X$, together with the weight map $\mu_L$ determined by  
the normalized
linearization of an {\em ample} line bundle $L$.  
In this case the minimum and maximum of the weight map $\mu_L$ are achieved precisely at the  sink and the  source of the action. We then denote by 
$$
0=a_0<\dots<a_r,$$ 
the weights $\mu_L(Y)$, $Y\in\cY$,  
and set: 
$$
Y_i:=\bigcup_{\mu_L(Y)=a_i}Y,
$$
so that $Y_0,Y_r$ are the sink $Y_-$  and the source $Y_+$ of the action on $X$, respectively.

The values $a_i$ are called the {\em critical values}, and the numbers $r$ and $\delta:=a_r$ are called the {\em criticality} and the {\em bandwidth} of the action on $(X,L)$, respectively.

\subsection{Equalized actions}\label{sssec:equal}

A $\C^*$-action  is {\em  equalized} at $Y\in\cY$ if for every $x\in \big(X^{-}(Y)\cup X^{+}(Y)\big)\setminus Y$ the isotropy group of the action at $x$ is trivial. Equivalently (see \cite[Lemma 2.1]{WORS3}), $\C^*$ acts on $N^+(Y)$ with all the weights equal to $+1$ and on $N^-(Y)$ with all the weights equal to $-1$. The action is called {\em equalized} if it is equalized at every $Y \in \cY$.

For an equalized action the closure of any $1$-dim\-ens\-ional orbit is a smooth rational curve.  In fact, if $x_-$ is the sink of a $1$-dimensional orbit $\C^*x$ and we denote by $Y\subset X$ the fixed point component containing $x_-$, then the BB-cell $X^-(Y)$ is locally $\C^*$-equivariantly isomorphic to $N^-(Y)$ around $x^-$. Since the action of $\C^*$ is equalized, the closures of the $1$-dimensional orbits of $\C^*$ in $N^-(Y)$ are lines in the fibers of $N^-(Y)\to Y$; in particular $\ol{\C^*x}$ is smooth at $x_-$; a similar argument applies to the source of the orbit. 

The $L$-degree of the closure of a $1$-dimensional orbit may be computed in terms of the weights at its extremal points. (Cf. \cite[Lemma 2.2]{WORS3}, \cite[Corollary 3.2]{RW}): 

\begin{lemma}[AM vs. FM]\label{lem:AMvsFM}
Let $(X,L)$ be a polarized pair with an equalized $\C^*$-action, and let $C$ be the closure of a $1$-dimensional orbit, with sink  $y_-$ and source $y_+$. Then $C$ is a smooth rational curve of $L$-degree  equal to $\mu_L(y_+)-\mu_L(y_-)$.
\end{lemma} 

Note that since every Cartier divisor on a projective variety $X$ is linearly equivalent to the difference of two ample divisors, we get the following:

\begin{corollary} \label{cor:can} 
Let $X$ be a smooth projective variety  with an equalized $\C^*$-action, $L$ a line bundle on $X$, and $C$ the closure of a $1$-dimensional orbit, with sink  $y_-$ and source $y_+$. Then:
\[L\cdot C=\mu_L(y_+)-\mu_L(y_-).\]
In particular, if the fixed point components containing the sink and the source of $C$ are denoted by $Y,Y'$, respectively, we have that:
 \begin{equation}\label{eq:can}
-K_X \cdot C 
=\big(\nu^+(Y')-\nu^-(Y')\big)-\big(\nu^+(Y) -\nu^-(Y)\big).
\end{equation}
\end{corollary}

\begin{remark}\label{rem:strat2}
Not even the equalization hypothesis is enough to guarantee that the BB-decompositions associated with a $\C^*$-action on a smooth projective variety are stratifications, see \cite[Example~6.1]{WORS6}.
\end{remark}

\subsection{Geometric quotients}\label{sssec:quotients}

A $\C^*$-action on a polarized variety $(X,L)$ allows us to define projective GIT quotients of $X$, as follows. Set, for every $\tau \in [0,\delta]\cap \Q$:
 \[A(\tau):=\bigoplus_{\substack{m\geq 0\\m\tau\in\Z}}\HH^0(X,mL)_{m\tau},\qquad \GX(\tau):=\Proj(A(\tau)).
\]
As shown in \cite[Section 2.2]{WORS6}, $\GX(\tau)$ is the same for any  $\tau\in (a_i,a_{i+1})\cap \Q$, and so we set:
\begin{itemize}[itemsep=5pt,topsep=5pt]
\item $\GX_{i,i}:=\GX(\tau)$ for $\tau=a_i$; 
\item $\GX_{i,i+1}:=\GX(\tau)$ for $\tau\in(a_i,a_{i+1})\cap \Q$.
\end{itemize}
Note that $\GX_{i,i+1}$ is a geometric quotient of $X$ for every $i$, while $\GX_{i,i}$ is only semigeometric. By construction, we have that
$$
\GX_{0,0}=Y_0,\quad\mbox{and}\quad \GX_{r,r}=Y_r. 
$$
 
\begin{remark}\label{rem:Btype}
If the $\C^*$-action on $X$ is equalized, the geometric quotients  $\GX_{0,1}$ and $\GX_{r-1,r}$ can be described as follows  (see \cite[Section 2.1.3]{WORS5} for details):
the $\C^*$-action on $X$ extends  to an action on the blowup $X^\flat$ of $X$ along its sink and source. Then $\GX_{0,1}$ and $\GX_{r-1,r}$ are isomorphic to the sink and the source of $X^\flat$, that is to $\PP(N_{Y_0,X}^\vee)$ and $\PP(N_{Y_r,X}^\vee)$, respectively (see \cite[Section 2.1.3]{WORS5}).
In particular, in the case of an action of {\em B-type}, i.e., when the extremal fixed point components are divisorial (equivalently, if $\nu^-(Y_-)=\nu^+(Y_+)=1$), we have that $\GX_{0,1}\simeq \GX_{0,0}\simeq Y_0$ and $\GX_{r-1,r}\simeq \GX_{r,r}\simeq Y_r$.
\end{remark}
 
The varieties $\GX_{i,i+1}$ are birationally equivalent, and the natural birational maps among them fit in the following  commutative diagram:
\begin{equation}
\begin{tikzcd}[
  column sep={3.1em,between origins},
  row sep={3em,between origins},
]
&&Z_{0,2}  \arrow[rd,"\prur"] \arrow[ld,"\prul",labels=above left]&&& |[xshift=-2.2em]|\dots &&  Z_{r-2,r} \arrow[rd,"\prur"] \arrow[dl,"\prul",labels=above left]\\ 
&\GX_{0,1} \arrow[rr,dashed] \arrow[rd] \arrow[dl]&&\GX_{1,2} \arrow[dl]&|[xshift=1.3em]|\dots&&\GX_{r-2,r-1} \arrow[rr,dashed] \arrow[rd] && |[xshift=-1.2em]|\GX_{r-1,r}\arrow[rd] \arrow[dl]&\\
\GX_{0,0}&&\GX_{1,1}&&&|[xshift=-2em]|\dots &&\GX_{r-1,r-1}&&|[xshift=-0.7em]|\GX_{r,r}
\end{tikzcd}
\label{eq:GITquot}
\end{equation}
The diagonal maps between geometric quotients are contractions, the maps $s$ and $d$ are smooth blow-ups, the horizontal maps  are Atiyah flips,  and the squares are defined as fiber products  (see \cite[Theorem 1.9]{Thaddeus1996}).

In particular, in the case in which the action is of B-type, the diagram above provides a birational map $\psi:Y_0 \dashrightarrow Y_r$.

The varieties $Z_{i,i+2}$ appearing in Diagram (\ref{eq:GITquot}) have been interpreted  in \cite{WORS6} in terms of Chow quotients of birational modifications of $X$ and the centers of the blowups $s$ and $d$ have been described. In this paper we will only use them as resolutions of the birational maps $\GX_{i,i+1}\dashrightarrow \GX_{i+1,i+2}$. More precisely, let us summarize in the next proposition what we will need later.

\begin{proposition} \label{prop:specialGIT}
The center of the blowup $s:Z_{i,i+2} \to \GX_{i,i+1}$ (resp. $d:Z_{i,i+2} \to \GX_{i+1,i+2}$) is  the inverse image in $\GX_{i,i+1}$ (resp. $\GX_{i+1,i+2}$) of the embedding of $Y_{i+1}$ into $\GX_{i+1,i+1}$.  Moreover, if $\nu^-(Y_{i+1})=1$ (resp. if $\nu^+(Y_{i+1})=1$) then $s$ (resp. $d$) is an isomorphism.
\end{proposition}

\begin{proof}
The first part of the statement follows from  \cite[Remark 5.10]{WORS6},  the second from \cite[Proof of Corollary 7.1]{WORS6}.
  \end{proof}

The following consequence of this result will be used in the study of equalized actions of criticality two with an isolated extremal fixed point (see Section \ref{sec:low}).

\begin{corollary}\label{cor:specialbw3} Assume that we have a B-type action of criticality three on $X$ with divisorial sink and source $Y_0,Y_3$ and that $\nu^+(Y_1)=\nu^-(Y_2)=1$.
Then the birational map $\psi:Y_0  \dashrightarrow Y_3 $ is resolved by a smooth blowup with center isomorphic to $Y_1$ and a smooth blowdown with center isomorphic to $Y_2$.
\end{corollary}

\begin{proof}
The fact that $\nu^+(Y_1)=1$ implies that every orbit with sink in $Y_0$ and source in $Y_1$ is uniquely determined by its source, hence the map $\GX_{0,1}\to \GX_{1,1}$ is an isomorphism. Analogously, since $\nu^-(Y_2)=1$,  $\GX_{2,3}\to \GX_{2,2}$ is an isomorphism, as well. By Proposition \ref{prop:specialGIT}, also $d:Z_{0,2}\to\GX_{1,2}$ and $s:Z_{1,3}\to\GX_{1,2}$ are isomorphisms. 

Under the B-type assumption we may also apply Remark \ref{rem:Btype}, and conclude that Diagram (\ref{eq:GITquot}) reduces to
\[
\begin{tikzcd}[
  column sep={3.4em,between origins},
  row sep={3em,between origins},
]
&&\GX_{1,2}  \arrow[rd,equal] \arrow[ld,"\prul",labels=above left]&& \GX_{1,2}  \arrow[rd,"\prur"] \arrow[ld,labels=above left,equal]&\\ 
&Y_0 \arrow[rd,equal] \arrow[dl,equal]&&\GX_{1,2} \arrow[rd] \arrow[dl]&&Y_3\arrow[rd,equal] \arrow[dl,equal] &\\
Y_0&&\GX_{1,1}&\ \ &\GX_{2,2}&&Y_3
\end{tikzcd}
\]
The description of the centers of $s$ and $d$ follow now from Proposition \ref{prop:specialGIT}.
  \end{proof}

\subsection{Action on rational homogeneous varieties}\label{ssec:RH}

Standard examples of $\C^*$-actions can be constructed easily on rational homogeneous (RH, for short) varieties, i.e., projective varieties that are quotients of a semisimple group $G$ by a parabolic subgroup $P$. Without loss of generality we assume that $G$ is the adjoint group of a semisimple Lie algebra $\fg$. Following \cite[Section 2.2]{MOSWW}, an RH variety is completely determined by the Dynkin diagram $\cD$ of $G$ marked on $\rho(G/P)$ nodes; we will then use the notation $\cD(I)$ to denote the RH variety determined by the Dynkin diagram $\cD$ marked in the set $I$ of nodes. Table \ref{tab:ihs} below illustrates how this notation works in the case of irreducible Hermitian symmetric spaces; these are the RH varieties $G/P$ of Picard number one such that $P$ acts irreducibly on the tangent space $T_{G/P,eP}$, which is $P$-equivariantly isomorphic to the quotient of the associated Lie algebras, $\fg/\fp$. 
\begin{table}[h!!]\renewcommand{\arraystretch}{1.1}
\centering
\begin{tabular}{|c|l|}
\hline
$\cD(I)$&Description\\
\hline
\hline
$\DA_n(1)$&$\PP^n$\\\hline
$\DA_n(k)$&Grassmannian of $(k-1)$ dimensional subspaces of $\PP^n$\\\hline
$\DB_n(1)$&$(2n-1)$-dimensional smooth quadric\\\hline
$\DD_n(1)$&$(2n-2)$-dimensional smooth quadric\\\hline
$\DD_n(n)$&$\big(n(n-1)/2\big)$-dimensional spinor variety\\\hline
\multirow{ 2}{*}{$\DC_n(n)$}&Lagrangian Grassmannian of $n$-dimensional subspaces\\
& of a $2n$-dimensional symplectic vector space \\\hline
$\DE_6(1)$& Cayley plane \\\hline
$\DE_7(7)$& Irreducible Hermitian symmetric space of type $\DE_7$\\\hline
\end{tabular}
\caption{Irreducible Hermitian symmetric spaces.\label{tab:ihs}}
\end{table}

Given an algebraic subgroup $\C^*\to G$, we get a faithful torus action on any RH variety $G/P$; without loss of generality, we may assume that $\C^*$ is contained in $P$, so that $eP\in (G/P)^{\C^*}$. Up to conjugation, faithful $\C^*$-actions are determined by gradings of the Lie algebra $\fg$.

Let us consider Cartan and Borel subgroups $H\subset B$ of $G$ contained in $P$, denote by $\fh\subset \fb\subset \fg$ the corresponding Lie subalgebras, by $\Phi$ the root system of $\fg$, and by $\Delta=\{\alpha_1,\dots,\alpha_n\}$ a base of positive simple root of $\Phi$, so that we have a Cartan decomposition: 
$$
\fg=\fh\oplus\bigoplus_{\alpha\in \Phi}\fg_\alpha.
$$ 
We then get $\Z$-gradings of $\fg$ by assigning integral weights to the roots of $\fg$, and these gradings determine $\C^*$-actions on $\fg$ that determine homomorphisms $\C^*\to G$. 

For instance, given $i=1,\dots, n$, we have a $\Z$-grading $\sigma_i$ of $\fg$ given by:
$$
\fg_{k}:=\bigoplus_{\substack{\alpha=\sum_{j=1}^nm_j\alpha_j\in\Phi\\m_i=k}}\fg_\alpha
$$
The $\C^*$-action determined by $\sigma_i$ on the variety $\cD(I)=G/P$ satisfies that $eP$ lies in the sink of the action. Furthermore,  $\{eP\}$ is the isolated sink of the action if and only if $I=\{i\}$.

Such a $\C^*$-action is equalized if and only if the induced grading is short, i.e. the only non-zero graded pieces of $\fg$ appear in degree $-1,0,1$ (\cite[Lemma 5.2]{Frantesi}). 
It turns out that, up to conjugation, equalized $\C^*$-actions are determined by some $\Z$-gradings of the form $\sigma_i$ (cf. \cite[p. 42]{Tev}, see Table \ref{tab:short}).
\begin{table}[ht!]\renewcommand{\arraystretch}{1.1}
\centering
\begin{tabular}{|c||c|c|c|c|c|c|}
\hline
	$\fg$ & $\DA_n$ & $\DB_n$ & $\DC_n$ & $\DD_n$ & $\DE_6$ & $\DE_7$ \\
	\hline\hline
	$\sigma_i$ & $\sigma_i$ for $i=1,\ldots,n$ & $\sigma_1$ & $\sigma_n$ & $\sigma_1,\sigma_{n-1},\sigma_n$ & $\sigma_1,\sigma_6$ & $\sigma_7$\\\hline
\end{tabular}
\caption{Short gradings of simple Lie algebras. \label{tab:short}}
\end{table}

We can then observe that equalized $\C^*$-actions with isolated sink occur precisely on irreducible Hermitian symmetric spaces. Table \ref{tab:hermite2} contains some  properties of these actions that we will use later on, more concretely their bandwidths, criticalities, sources, and fixed point components of weight $1$. 
\begin{table}[h!!]\renewcommand{\arraystretch}{1.1}
\centering
\begin{tabular}{|c|c|c|c|c|}
\hline
$\cD(I)$& \!grading\!&$\delta=a_r=r$& $Y_1$& $Y_r$ \\
\hline
\hline
$\DA_n(1)$&$\sigma_1$&$1$& $\DA_{n-1}(1)$& $\DA_{n-1}(1)$\\\hline
$\!\DA_n(k), \,\,(k\leq \frac{n+1}{2})\!\!$&$\sigma_k$&$k$&\!$\DA_{k-1}(1)\times \DA_{n-k}(1)$\!&$\DA_{n-k}(k)$\\\hline
$\DB_n(1)$&$\sigma_1$&$2$&$\DB_{n-1}(1)$&pt\\\hline
$\DD_n(1)$&$\sigma_1$&$2$&$\DD_{n-1}(1)$&pt\\\hline
$\DD_n(n),\,\,(n\mbox{ even})$&$\sigma_n$&$n/2$&$\DA_{n-1}(2)$&pt\\\hline
$\DD_n(n),\,\,(n\mbox{ odd})$&$\sigma_n$&$\lfloor n/2\rfloor$&$\DA_{n-1}(2)$&$\DA_{n-1}(n-1)$\\\hline
$\DC_n(n)$&$\sigma_n$&$n$&$\DA_{n-1}(1)$&pt\\\hline
$\DE_6(1)$&$\sigma_1$&$2$&$\DD_5(5)$&$\DD_5(1)$ \\\hline
$\DE_7(7)$&$\sigma_7$&$3$&$\DE_6(1)$&pt\\\hline
\end{tabular}
\caption{Irreducible Hermitian symmetric spaces and  $\C^*$-actions with isolated sink on them.  \label{tab:hermite2}}
\end{table}

\subsection{Families of rational curves}\label{ssec:ratcurves}

For the readers' convenience, we will introduce here some notation that we will use along the paper, regarding rational curves on algebraic varieties. We refer to \cite{kollar} for a complete account on this topic.
 
Given a smooth complex projective variety $X$, a {\it family of rational curves on} $X$ is an irreducible component $\cM$ of the normalization $\rat^n(X)$ of the scheme $\rat(X)$. Such a family comes equipped with a smooth $\P^1$-fibration $p:\cU\to \cM$, and an evaluation morphism $q:\cU \to X$. 
The family $\cM$ is called {\it dominating} if $q$ is dominant, and {\it unsplit} if $\cM$ is proper. We will denote by $\cMU_x$ the fiber of $q$ over the point $x \in X$, and by $\cM_x:=p(q^{-1}(x))\subset\cM$ the family of rational curves of $\cM$ passing by $x$.

In this paper we will consider the case in which $\cM$ is a dominating, unsplit family such that $q:\cU\to X$ 
is a smooth morphism, and we will simply say that the family is {\it beautiful}. 

\begin{remark}\label{rem:beauty}
By \cite[II.Proposition~3.5]{kollar}, given an unsplit family of rational curves $\cM$, it is beautiful if and only if the following property is fulfilled: 
$\nu^*T_X$ is globally generated -- equivalently nef -- for the normalization $\nu:\P^1\to X$ of any curve of the family. From this it follows that $\cM$, $\cU$ and $\cM_x$, for every $x\in X$, are smooth, and that, for every element $C$ of $\cM$: 
\begin{equation}\label{eq:dimMx} \dim \cM_x  = -K_X \cdot C-2.
\end{equation}
In particular, $\rat^n(X)$ is smooth at every point of $\cM$, and every deformation of a curve of $\cM$ belongs to $\cM$. Then, given any irreducible subvariety  $Y\subset X$, the set of elements of $\cM$ parametrizing curves of the family contained in $Y$ is a union of irreducible components of $\rat^n(Y)$.
\end{remark}

\begin{remark}\label{rem:MU}
Note that, if $\cM$ is a beautiful family of rational curves in $X$ and $X$ is a Fano manifold (this will always be the case if the Picard number of $X$ is one), then  every $\cU_x$ is connected, therefore irreducible (cf. \cite[Proposition~1.6]{Kan5}). 
Moreover, it is known that, via the morphism $p$, $\cU_x\simeq \cM_x$ for {\it general} $x\in X$ (cf. \cite[Lemma~2.10]{MOSWW}). 
\end{remark}

Given an unsplit dominating family of rational curves $\cM$ we can define a relation of {\em rational chain-connectedness with respect to $\cM$} on $X$ in the following way: two points $x$ and $y$ of $X$ are in rc$(\cM)$-relation if
there exists a chain of curves parametrized by $\cM$ which joins
$x$ and $y$. By \cite[Proposition 1.1]{AW} the variety $X$ is rc$(\cM)$-connected if and only if its Picard number is one.
In particular, since, for a beautiful family the fibers of the evaluation map $q$ are irreducible, we may state the following technical Lemma, that will be used later on.

\begin{lemma}\label{lem:trivial} Let $\cM$ be a beautiful family of rational curves on a smooth variety $X$ with Picard number one.
If $Z \subset X$ is a proper irreducible closed subset, then a general curve of $\cM$ meeting $Z$ is not contained in $Z$. 
\end{lemma}

\begin{proof}
Curves of $\cM$ meeting $Z$ are parametrized by $\cM_Z:=p(q^{-1}(Z))$, which is irreducible. If $q^{-1}(Z) \cap p^{-1}
(\cM_Z) = q^{-1}(Z)$ every curve of $\cM$ meeting $Z$ would be contained in $Z$, but this is not possible since $X$ is rc($\cM$)-connected. 
Therefore $q^{-1}(Z) \cap p^{-1}(\cM_Z)$ is a proper closed subset of $p^{-1}(\cM_Z)$, and so it does not contain a general fiber of $p^{-1}(\cM_Z) \to \cM_Z$. 
  \end{proof}

Remarkably, the homogeneity of a Fano manifold can be detected by finding a beautiful family of rational curves with certain properties. The following statement, that we will use later on, has been proved in \cite{OSWi}.

\begin{theorem}\label{thm:OSWi}
Let $X$ be a Fano manifold of Picard number one, and let $p:\cU \to \cM$ be a beautiful family of rational curves in $X$. If the  fiber $\cU_x=q^{-1}(x)$ is a rational homogeneous variety for every $x\in X$, then $X$ is rational homogeneous. 
\end{theorem}

\begin{example}\label{ex:beautifulHermite}
In the spirit of the above statement, irreducible Hermitian symmetric spaces can be characterized by the homogeneous varieties $\cU_x=q^{-1}(x)$, which in this case is isomorphic to the variety $\cM_x$ parametrizing lines passing by $x$, for every $x$. 
Table \ref{tab:hermite3}  describes $\cM_x$ for every type of irreducible Hermitian symmetric space. 
\begin{table}[h!!]\renewcommand{\arraystretch}{1.1}
\centering
\begin{tabular}{|c|c|c|c|c|}
\hline
$\cD(I)$& $\cM_x\simeq \cU_x$ \\
\hline
\hline
$\DA_n(1)$&$\DA_{n-1}(1)$\\\hline
$\!\DA_n(k)\!$&\!$\DA_{k-1}(1)\times \DA_{n-k}(1)$\!\\\hline
$\DB_n(1)$&$\DB_{n-1}(1)$\\\hline
$\DD_n(1)$&$\DD_{n-1}(1)$\\\hline
$\DD_n(n)$&$\DA_{n-1}(2)$\\\hline
$\DC_n(n)$&$\DA_{n-1}(1)$\\\hline
$\DE_6(1)$&$\DD_5(5)$ \\\hline
$\DE_7(7)$&$\DE_6(1)$\\\hline
\end{tabular}
\caption{Families of lines by a point in irreducible Hermitian symmetric spaces.\label{tab:hermite3}}
\end{table}

Note that, by comparing this with Table \ref{tab:hermite2}, we immediately see that, in each case, $\cM_x$ coincides with the component $Y_1$ of fixed points of weight one of the corresponding $\C^*$-action with isolated sink. 

\end{example}

\section{Actions with small criticality and isolated sink/source}\label{sec:low}

In this section, which can be considered as a complement to the results presented in \cite{WORS1}, we study smooth projective varieties $X$ admitting an equalized $\C^*$-action of small criticality, with an isolated extremal fixed point. By composing a $\C^*$-action with the inversion map of $\C^*$, we may always assume that the isolated fixed point is the sink of the action.

The classifications obtained  will be later used in the proof of Theorem \ref{thm:CP1}. With the notation introduced in the previous section, we will consider an ample line bundle $L$ on $X$, and extend the action on $X$ to the pair $(X,L)$.

Let us start by showing that, if an equalized $\C^*$-action  has an isolated extremal fixed point, then the space of $1$-cycles modulo numerical equivalence is generated by the numerical classes of some orbits.

\begin{definition}(Cf. \cite[p. 1443]{RW}) 
An orbit with sink in $Y\in \cY$ and source in $Y'$ ($Y'\neq Y$) is called  {\em minimal} if there is no chain of closures of $1$-dimensional orbits of length $>1$ joining $Y$ and $Y'$. 
\end{definition}

\begin{proposition}\label{prop:bound}
Let $(X,L)$ be a smooth polarized pair endowed with an equalized $\C^*$-action. 
Then every curve in $X$ is numerically equivalent to a linear combination with rational coefficients of the class of a curve contained in the sink and classes of closures of minimal orbits. In particular, if the action has an isolated extremal fixed point, then $\Nu(X)$ is generated by classes of closures of minimal orbits. 
\end{proposition}

\begin{proof} 
In this proof we will use the Chow quotient $\CX$ of $X$ by the $\C^*$-action, which is the projective variety defined as the normalization of the closure in $\Chow(X)$ of the variety parametrizing general $\C^*$-invariant $1$-cycles in $X$. We refer the interested reader to \cite[Section 4]{WORS6} for details. Let us denote its universal family by:
\[
\xymatrix{
\cU \ar[r]^q\ar[d]_p& X \\
\CX&}
\]
Let $C$ be an irreducible curve in $X$ which is not the closure of a $1$-dimensional  orbit. Then $q^{-1} (C)$ contains an irreducible curve $C'$ which is not
contained in a fiber of $p$ and dominates $C$ via $q$.
Set $B = p(C')$ and let $S$ be the surface $ p^{-1}(B)$.\\
Since the image of every fiber of $p$ meets the sink, 
the surface $S$ contains a curve $C''$ which dominates $B$ and such that $q(C'')$ is contained in the sink.

By \cite[II.4.19]{kollar}, every curve in $S$ is algebraically equivalent to a linear combination
with rational coefficients of $C''$, and of the irreducible components of fibers of $p_{|S}$
(in the quoted statement take $X = S$, $Y= B$ and $Z = C''$).

Thus any curve in $q(S)$, and in particular $C$, is algebraically, hence numerically,
equivalent 
to a linear combination with rational coefficients of $q_*(C'')$ and of
closures of $1$-dimensional orbits. 
We conclude observing that, by Corollary \ref{cor:can}, 
the numerical class of the closure of an orbit is a sum of numerical classes of closures of minimal orbits.
  \end{proof}

\subsection{Actions of criticality one with an isolated extremal fixed point}\label{ssec:critone}

The following statement is an extension of \cite[Theorem~3.1]{RW}.

\begin{proposition}\label{prop:crit1}
Let $X$ be a smooth projective variety with an equalized $\C^*$-action of criticality one and isolated sink. Then $X$ is a projective space of the form $\P(V_0\oplus V_1)$, where $V_0,V_1$ are vector spaces on which $\C^*$-acts with weight $0$ and $1$, respectively, and $\dim(V_0)=1$. 
\end{proposition}

\begin{proof} Note first that, since the criticality is $1$, we only have two fixed point components, $Y_0,Y_1$, and every $1$-dimensional orbit is minimal. Since moreover the sink is an isolated point,  by Proposition \ref{prop:bound} the Picard number of $X$ is one; in particular every effective divisor is ample. On the other hand, $\nu^+(Y_1)$ must be equal to one, otherwise the Bend-and-Break Lemma applied to closures of $1$-dimensional orbits passing through a point in the source would contradict the fact that the action is equalized, and that every $1$-dimensional orbit is minimal (cf.  \cite[Proof of Lemma 4.5]{WORS3}). 

In particular $Y_1$ must be an effective -- hence ample -- divisor. Applying Remark \ref{rem:Btype}, we get that $\P^{\dim X-1}\simeq\P(N_{Y_0,X}^\vee)\simeq \GX_{0,1}\simeq Y_1$. We conclude that $X \simeq \P^{\dim X}$, since it contains an ample divisor isomorphic to a projective space. 

Finally, we note that the equalized $\C^*$-actions of projective spaces $\P(V)$ are known to be induced  by decompositions of $V$ as a sum of two eigenspaces of weights $0,1$ (see, for instance, \cite{FS1}). Since the sink of $X$ is zero-dimensional, we conclude that $\dim(V_0)=1$.
  \end{proof}

Now we will study equalized actions of criticality two with an isolated fixed point; we will consider two different cases, according to the Picard number of $X$.

\subsection{Actions of criticality 2 and an isolated extremal fixed point (I)}\label{ssec:crittwo1}

If $X$ has Picard number one, we will relate the action to a {\em special birational transformation} of type $(2,1)$.
These are birational transformations $\psi:\PP^n\dashrightarrow Z \subset \P^N$, where $Z$ is a smooth variety of Picard number one, given by a linear system of degree two,  such that $\psi^{-1}$ is given by a linear system of degree one, and the base locus of $\psi$ is smooth and connected.

In \cite[Section 5.1]{WORS5} it was explained how to realize special birational transformations of type (2,1) as birational transformations associated with some equalized $\C^*$-actions. The following Lemma is a partial converse to that construction.

\begin{lemma}\label{lem:21}
Let $\pi: \cX \to \P^1$ be a smooth morphism, with fibers of Picard number one, and assume that there exists an equalized $\C^*$-action on $\cX$, which is $\pi$-equivariant, such that
\begin{enumerate}
\item The action on $\PP^1$ is the standard one, with sink $\infty$ and source $0$; 
\item The action on $X_\infty=\pi^{-1}(\infty)$ is of criticality two with respect to the ample generator of $\Pic(X_\infty)$, and has isolated sink $Y_0=\{y_0\}$;
\item The action on $X_0=\pi^{-1}(0)$ is trivial.
\end{enumerate} 
Let $\cX^\flat$ be the blowup of $\cX$ along $y_0$, with exceptional divisor $P \simeq  \PP^{\dim X_0}$.  Then the induced birational map $\psi: P \dashrightarrow X_0$ is a special birational transformation of type $(2,1)$.
\end{lemma}

\begin{proof} Denote by $Y_0\cup Y_1\cup Y_2$ the fixed point locus of  $X_\infty$, where $Y_0=\{y_0\},Y_2$ is the source, and $Y_1$ is the inner fixed point locus. 
Note that, by assumption,  $\cX^{\C^*}=Y_0\cup Y_1\cup Y_2\cup X_0$. Since $Y_0$ is an isolated point, it follows (as in the proof of Proposition \ref{prop:crit1}) that $\nu^+(Y_1)=1$ and, since $X_\infty=B^+(Y_2)$ is a divisor, we have that $\nu^-(Y_2)=1$. Now consider an ample line bundle $L$ on $\cX$. Its restriction to $X_\infty$ is ample, therefore it is a multiple of the ample generator of $\Pic(X_\infty)$, and we may claim that $X_\infty$ has criticality $2$ with respect to $L_{|X_\infty}$. Since $\mu_L(Y_2)<\mu_L(X_0)$ (in fact, $X_0$ is the source of the action on $\cX$, so it is the only fixed point component of maximal weight), it then follows that the criticality of the action on $(\cX,L)$ is $3$.
 
Let now $b:\cX^\flat\to \cX$ be the blowup of $\cX$ at $y_0$, endowed with the B-type $\C^*$-action inherited from $\cX$, whose sink is the exceptional divisor $P$; the remaining fixed point locus is isomorphic to $Y_1\cup Y_2\cup X_0$ so, abusing notation, we denote it by $Y_1\cup Y_2\cup X_0\subset \cX^\flat$. Polarizing $\cX^\flat$ with an ample line bundle of the form $L^\flat:=mb^*(L)-P$, $m\gg 0$, we may assert that the $\C^*$-action on $(\cX^\flat,L^\flat)$ has criticality three, and it satisfies all the assumptions of Corollary \ref{cor:specialbw3}. 
We may then claim that the  birational map $\psi:P\dashrightarrow X_0$ induced by the action is the composition of a smooth blowup $s: W:=\cG\!\cX_{1,2} \to P$ with center isomorphic to $Y_1$ and exceptional divisor $E_1$, and a smooth blowdown $d: W \to X_0$ with center isomorphic to $Y_2$ and exceptional divisor $E_2$:
\begin{equation}\label{eq:diagbi}\xymatrix@R=30pt@C=20pt{ & E_1 \ar[ld] \ar@{^(->}[r]& \ar[ld]_{s}W \ar[rd]^{d} & E_2 \ar[rd] \ar@{_(->}[l]& \\
Y_1\ar@{^(->}[r]&P \ar@{-->}@/^1pt/[rr]^{\psi}& &X_0 
&\ar@{_(->}[l] Y_2}
\end{equation}
The image in $P$ of $E_2$ is $B^+(Y_2) \cap P$, which can be described as the set of tangent directions (at $y_0$) of orbits converging to $y_0$ inside $X_\infty$, that is, as the hyperplane $\P(T_{X_\infty,y_0})$ (which is  the exceptional divisor of the blowup $X_\infty^\flat$ of $X_\infty$ along $y_0$). 
In particular, if $\ell$ is a general line in $P$ and $\tl \ell$ is its strict transform in $W$, then $E_2 \cdot \tl \ell = 1$. On the other hand, obviously, $E_1\cdot\tl \ell=0$.

Denote by $H_1$ and $H_2$ the ample generators of the Picard groups of $P$ and $X_0$, respectively, and let $m_1, m_2$ be the positive integers such that 
\[d^*H_2= m_1s^*H_1-E_1, \quad  s^*H_1= m_2d^*H_2-E_2\]
Since $s^*H_1\cdot \tl \ell=1$, intersecting with $\tl \ell$ we  get  
$$d^*H_2\cdot\tl \ell=m_1,\quad 1=m_2d^*H_2\cdot\tl \ell-1,\quad\mbox{therefore }m_1m_2=2.$$
 Since $\psi$ is birational we must have $m_1 >1$, hence the only possibility is $(m_1,m_2)=(2,1)$. This finishes the proof.
  \end{proof}
The following lemma shows how to construct a variety satisfying the assumptions of Lemma \ref{lem:21} upon a variety with an equalized action of criticality two with isolated sink.

\begin{lemma}\label{lem:deg} Let $X$ be a smooth variety, and $\alpha: \C^* \to \Aut(X)$ a $\C^*$-action. Consider the standard $\C^*$-action on the projective line $\P^1$, with sink $\infty$ and source $0$. 
There exist a smooth projective variety $\cX$ endowed with a $\C^*$-action $\tl \alpha$, and a smooth $\C^*$-equivariant fibration $\pi:\cX \to \P^1$ whose fibers $X_t:=\pi^{-1}(t)$ are isomorphic to $X$, such that:
\begin{itemize}
\item the action $\tl \alpha$ restricted to $X_\infty$ is $\alpha$.
\item the action $\tl \alpha$ restricted to $X_0$ is trivial.
\end{itemize}
Moreover, if $\alpha$ is equalized, then we may also require that $\tl \alpha$ is equalized. 
\end{lemma} 

\begin{proof}
Set $U_0:=\P^1  \setminus \{\infty\}$, $U_\infty:=\P^1  \setminus \{0\}$, and consider the map 
\[
\xymatrix@R=3pt{U_0 \times X & U_{\infty} \times X\\
\cup & \cup\\
\C^* \times X \ar[r]^{\phi_\alpha} & \C^* \times X\\
(t,x) \ar@{|->}[r] &(t, \alpha_t(x))\\
}
\]
We use it to glue $U_0 \times X$ with $U_\infty \times X$ along $\C^* \times X$, obtaining a locally trivial $X$-fibration $\pi:\cX\to\P^1$.
On $U_0 \times X$ we consider the action $\tl \alpha$, defined as $\tl \alpha_s(t,x):=(st, x)$, for $s\in \C^*$. It has a unique fixed point component, $X_0=\{0\}\times X$. 
We claim that this action extends to $\cX$ and that, restricted to $X_\infty\simeq X$, it coincides with the action $\alpha$.

In fact, using the definitions of $\phi_\alpha$ and $\tl \alpha$, one may easily see that the induced action on $\C^* \times X \subset U_{\infty} \times X$ is given by $\tl \alpha_s(t,x)=(st,\alpha_s(x))$, which clearly extends to $U_{\infty}\times X$ with the required properties.

The fixed point components of $\tl \alpha$ are $X_0$ and the fixed point components of the action $\alpha$ on  $X_\infty$. If $\alpha$ is equalized it is enough to observe that a $1$-dimensional orbit $\C^*x$ is
either contained in the fiber $X_\infty$, where
the action is assumed to be equalized, or it is of the form $(t, x) \subset \C^* \times X\subset U_\infty \times X$.
  \end{proof}

Using Lemma \ref{lem:21} and Lemma \ref{lem:deg}, together with the classification of  special birational transformations of type $(2,1)$ due to Fu and Hwang (cf. \cite{FuHw}) we obtain the following:

\begin{theorem}\label{thm:rho1} 
Let $(X,L)$ be a pair consisting of a smooth projective variety $X$  
and an ample line bundle $L$, with an equalized $\C^*$-action of criticality two and isolated sink. If the Picard number of $X$ is one then $X$ is one of the following:
\begin{enumerate}
\item a smooth quadric $\DB_n(1)$ or $\DD_n(1)$; 
\item a Grassmannian of lines in $\P^n$, $\DA_n(2)$;
\item the $10$-dimensional spinor variety $\DD_5(5)$;
\item the $16$-dimensional Cartan variety $\DE_6(1)$;
\item a general linear complete intersection of codimension
$k \le 2$ of the Grassmannian $\DA_4(2)$; 
\item a general linear complete intersection of codimension
 $k \le 3$ of $\DD_5(5)$. 
\end{enumerate}
\end{theorem}

\begin{proof}
We construct a $\C^*$-equivariant fibration $\pi:\cX\to \P^1$ as in Lemma \ref{lem:deg}, and blow it up along its sink $y_0$. The resulting variety has an equalized $\C^*$-action, with source isomorphic to $X$, satisfying the assumptions of Lemma \ref{lem:21}. Then the induced birational transformation $\psi:P\to X$ is special of type $(2,1)$, and we conclude using the main statement in \cite{FuHw}. Note that one of the cases in that statement  corresponds to $X=P$, that we have discarded here since it does not admit equalized actions of criticality two. 
  \end{proof}

The next statement, which extends \cite[Propositions~4.6,~4.7]{Li17}, tells us that in the above list the non-homogeneous examples can be distinguished from the rest by the fact that they do not contain a beautiful family of rational curves. 

\begin{proposition}\label{prop:FHnonhomo}
Let $X$ be a smooth projective variety endowed with a $\C^*$-action, satisfying the hypotheses of Theorem \ref{thm:rho1}. Then $X$ contains a beautiful family of rational curves  if and only if it is rational homogeneous.
\end{proposition}

\begin{proof}
The existence of beautiful families of rational curves in rational homogeneous varieties is well known. We will show that if $X$ is one of the non-homogeneous cases ($5$-$6$) of Theorem \ref{thm:rho1}, then it does not contain a beautiful family of rational curves.

Each of these two cases appears as a smooth linear section of a homogeneous variety -- the Spinor variety $\DD_5(5)$, 
or the Grassmannian $\DA_4(2)$ -- 
that we denote by $X'$. The $\C^*$-action on $X$ is the restriction of an equalized  $\C^*$-action of criticality two  on $X'$ with isolated sink. We denote the fixed point components of this action by $Y_0'=\{y_0\}$, $Y'_1$ and $Y'_2$. 

Denoting by $L'$ the (very) ample generator of $\Pic({X'})$, these fixed point components are the intersection with ${X'}$ of the fixed point components (of weights $0,1,2$) of the induced action on $\P(\HH^0({L'}))$. The variety $X$ can be described as the intersection $X={X'}\cap\P(\HH^0({L'})_0\oplus V \oplus \HH^0({L'})_2)$, where $V\subset \HH^0(L')_1$ is a general subspace of codimension $k$. In particular the fixed point component $Y_2\subset X$ is equal to $Y'_2$. The family ${\cL}$ of lines $X'$ is a beautiful family of rational curves; given $y\in Y'_2$, the induced $\C^*$-action  on the subfamily ${\cL}_y$ of lines passing by $y$ has two fixed point components, that we denote by ${\cL}_y^-$ (the sink, parametrizing lines which are orbits with source in $y$), and ${\cL}_y^+$ (the source, parametrizing lines contained in $Y_2'$ and passing by $y$). Table \ref{tab:fixedIHS}  describes these varieties in each case.
\begin{table}[h!!]\renewcommand{\arraystretch}{1.1}
\centering
\begin{tabular}{|c|c|c|c|c|c|}
\hline
$X'$&$Y'_1$&$Y'_2$&${\cL}_y$&$\cL_y^-$&${\cL}_y^+$\\
\hline
\hline
$\DD_5(5)$&$\DA_4(2)$&$\DA_4(4)\simeq\P^4$&$\DA_4(2)$&$\DA_{3}(2)$&$\DA_{3}(1)$\\\hline
$\DA_4(2)$&$\DA_1(1)\times\DA_2(1)$&$\DA_2(2)\simeq \P^2$&$\DA_1(1)\times\DA_2(1)$&$\DA_1(1)\times\DA_1(1)$&$\DA_{1}(1)$\\\hline
\end{tabular}
\caption{Fixed point components of the actions on $X'$ and $\cL_y$. \label{tab:fixedIHS}}
\end{table}

For every $y$, the subfamily ${\cL}_y^-$ can be identified with a subvariety of $Y'_1$ assigning to a line in ${\cL}_y^-$ its sink in $Y'_1$. In the embedding $Y'_1\subset \P(\HH^0({L'})_1)$ these subvarieties are quadrics of dimension $4$ and $2$, respectively. On the other hand, we may identify $\cL_y^-$ with a quadric contained in $\P(N_{Y'_2|X,y})$, for every $y\in Y_2$. Let us note that a standard computation provides:
$$
N_{Y'_2|X}=\left\{\begin{array}{ll}\bigwedge^2\big(T_{\P^4}(-1)\big)&\mbox{if }X'=\DD_5(5),\\[3pt]T_{\P^2}(-1)\otimes \cO_{\P^2}^{\oplus 2}&\mbox{if }X'=\DA_4(2),\end{array}\right.
$$
and so the family of quadrics $\cL_y^-$, $y\in Y'_2$, is given, in both cases, by a symmetric isomorphism $N_{Y'_2|X}\simeq N_{Y'_2|X}^\vee(1)$. Let us denote by $\cL^-\subset \P(N_{Y'_2|X})$ the union of all these quadrics; we have a commutative diagram:
$$
\xymatrix{\cL^-\ar[r]\ar@{^(->}+<0pt,-8pt>;[d]&Y'_1\ar@{^(->}+<0pt,-10pt>;[d]\\\P(N_{Y'_2|X})\ar[r]&\P(\HH^0({L'})_1)}
$$
Given a general hyperplane $H\subset \P(\HH^0({L'})_1)$, its pullback to $\P(N_{Y'_2|X})$ is given by a non-zero map $s:\cO_{Y'_2}\to N_{Y'_2|X}$. It meets all the quadrics $\cL_y^-$, and the rank of the intersection drops precisely on the set of zeros of the composition:
$$
\cO_{Y'_2}\to N_{Y'_2|X}\stackrel{\sim}{\lra} N_{Y'_2|X}^\vee(1) \to \cO_{Y'_2}(1),
$$
that is, on a hyperplane section of $Y'_2$.

Assume that $X$ has a beautiful family of rational curves $\cM$; since every curve in $X$ is algebraically equivalent to a $\C^*$-invariant cycle, it follows that $\cM$ is necessarily the family of lines in $X$. We will get to a contradiction by looking at the subfamily $\cM_y$ of lines in $X$ passing by a point $y\in Y_2$, and the fixed point component $\cM_y^-$, which must be smooth, irreducible, and of the expected dimension, if $\cM$ is beautiful  (see Remark \ref{rem:beauty}). 
The above argument tells us that for the points $y$ contained in a  linear subspace of codimension $k$ of $Y'_2$, the variety $\cM_y^-={\cL}_y^-\cap \P(V)$ is singular, a contradiction. 
  \end{proof}

\subsection{Actions of criticality 2 and an isolated extremal fixed point (II)}\label{ssec:crittwo2}

Let us finally deal with the case of higher Picard number. The only known examples of equalized $\C^*$-actions on smooth projective varieties of Picard number different from $1$, criticality two and isolated sink are the following:

\begin{example}\label{ex:c2} Let $\P^n\times \P^m$  be the product of two projective spaces, endowed with the product of two equalized actions of criticality one and isolated sink (see Proposition \ref{prop:crit1}). Polarizing $\P^n\times \P^m$ with the line bundle  $\cO(1,1)$, the action has criticality two and isolated sink. On the other hand, we may consider $\P^n$ endowed with the equalized action of criticality one and isolated sink, and denote by $X$ its blowup along {\em any} smooth  (not necessarily connected) subvariety of the source; one can show that there exists a polarization of  $X$ such that the criticality of the induced action on $X$ is two.
\end{example}

We believe that one can prove that these are in fact all the examples of this kind of actions. For our current purposes, it would be enough to classify the actions that satisfy a stronger assumption:

\begin{theorem}\label{thm:CP1rho2} Let $X$ be a smooth variety of Picard number greater than one, admitting an equalized $\C^*$-action with isolated sink.  
Assume that there exists a line bundle $L$ on $X$ that has positive degree on rational curves, and degree two on the closure of a general orbit.  
Then $X$ is either a product of two projective spaces or the blowup of the projective space along a linear subspace. In both cases $L$ is ample and has degree one on the curves which generate the extremal rays of $\NE(X)$.
\end{theorem}

\begin{proof} Denote by $Y_0=\{y_0\}$ the sink, by $Y_2$ the source, and by $Y_1$ the inner fixed point locus of the action, with irreducible components $Y_{1,i}$, $i=1,\dots,k$. The assumptions on $L$ allow us to guarantee that every $1$-dimensional orbit with a limit point in $Y_1$ has $L$-degree one (and so it is minimal), and that there are no closures of orbits linking two different components of $Y_1$. In fact, by means of Corollary \ref{cor:can}, our assumptions imply that $\mu_L(Y_2)-\mu_L(Y_0)=2$, and that given two components $Y,Y'$ linked by a connected chain of closures of $1$-dimensional curves $C_1\cup \dots\cup C_k$ (so that the sink of $C_1$ belongs to $Y$, the source of $C_k$ belongs to $Y'$, and the source of every $C_i$ is the sink of $C_{i+1}$), then $\mu_L(Y)<\mu_L(Y')$; then we infer that $\mu_L(Y_{1,i})=\mu_L(Y_0)+1$ for every $i$, and that there are no closures of orbits linking two different components of $Y_1$. In particular, two different $B^-(Y_{1,i})$'s can only meet along points of the source $Y_2$. 

Let us start by assuming that $Y_2$ is a divisor in $X$.
Since the sink is isolated, we know that $\nu^+(Y_1)=1$ 
(see \cite[Proof of Lemma 4.5]{WORS3}), and so $B^-(Y_1)$ is a union of irreducible divisors $E_i:=B^-(Y_{1,i})$, $i=1,\dots,k$. Moreover, since $Y_2$ is a divisor, we have that $\nu^-(Y_2)=1$, and so the union of the $E_i$'s is disjoint. Polarizing $X$ with an ample line bundle $L'$ assigns weights to the components $Y_{1,i}$, not necessarily equal. The arguments of \cite[Section~3]{WORS3} allow us to claim that there is a sequence of morphisms  $$X\to X_1 \to \dots \to X_r$$ that contract unions of divisors $E_i$, ordered decreasingly by the $L'$-weights of the $Y_{1,i}$'s. The final product is a variety $X':=X_r$ endowed with an equalized $\C^*$-action with isolated sink, divisorial source $Y_2'$, and no inner fixed point components. In particular, it has criticality one and we may assert, by Proposition \ref{prop:crit1}, that $X'$ is the projective space $\PP^{\dim X}$ and that $Y_2'$ is a hyperplane. We conclude that the contraction $\varphi:X\to X'$ is a blowup of $X'=\P^{\dim X}$ along a disjoint union of subvarieties of a hyperplane $H\subset \P^{\dim X}$.

A simple calculation shows that the only line bundle on $X$ which has degree two on the closure of a general orbit and degree one on closures of minimal orbits is $L=2\varphi^*H - \sum_{i=1}^k E_i$. In order for this line bundle to have positive degree in every rational curve,  the union of the centers of the blowup $\varphi$ cannot have proper secants, hence it is a linear space.

Assume now that the codimension of the source $Y_2$ in $X$ is at least two. Since we are assuming that the Picard number of $X$ is not $1$, by Proposition \ref{prop:bound}  there are closures of orbits $C_1$, linking $y_0$ and $Y_1$, and $C_2$, linking $Y_1$ and $Y_2$, which are not numerically proportional to the closure $C_g$ of the general orbit. Denote by $V^1$ and $V^2$ the normalization of the components  of $\rat(X)$ containing the elements parametrizing $C_1$ and $C_2$. Since $L \cdot C_i=1$ the families are unsplit.
Using Corollary \ref{cor:can} we compute
\[\begin{array}{r}
-K_X \cdot C_1  -K_X \cdot C_2 = \big(\mu_{-K_X}(Y_1)-\mu_{-K_X}(y_0)\big)-\big(\mu_{-K_X}(Y_2)-\mu_{-K_X}(Y_1)\big)=\\[3pt]=\mu_{-K_X}(Y_2)-\mu_{-K_X}(y_0)=-K_X \cdot C_g = 2 \dim X -\dim Y_2 \ge \dim X+2.
\end{array}
\]  
We can then apply \cite[Remark 5.5]{ACO} to get that that both $V^1$ and $V^2$ are covering families.
We conclude that $X$ is a product of two projective spaces applying \cite[Main Theorem]{O}. 
  \end{proof}

\section{Induced actions on families of rational curves}\label{sec:wfpc}

In this section we will work in the following situation:

\begin{setup}\label{set:basic}
$(X,L)$ is a smooth polarized pair, endowed with an equalized $\C^*$-action, and admitting a beautiful family of rational curves $\cM$, of $L$-degree $d$, with evaluation morphism $q: \cU \to \cM$ (see Section \ref{ssec:ratcurves}). 
\end{setup}

Our goal will be to describe the induced $\C^*$-actions on $\cM$ and $\cU$.

\subsection{Action on $\cM$}
Let $C$ be a curve such that  $[C] \in \cM$; the point $[C]$ is fixed by the induced action if and only if $C$ is a $\C^*$-invariant curve, that is, if either $C \subset Y$ with $Y \in \cY$, or $C$ is the closure of an orbit of the action. The latter curves will be called {\em bridges} of the family $\cM$ (with respect to the $\C^*$-action). If a bridge $C$ has sink and source in fixed components $Y, Y'$, we say that {\em the bridge $C$ links $Y$ and $Y'$}. Note that in this case, by Lemma \ref{lem:AMvsFM}, $\mu_L(Y) \not =\mu_L(Y')$, hence $Y \not= Y'$. 

\begin{definition}\label{def:limitcurves}
If $[C] \in \cM$ is not a fixed point of the induced action on $\cM$ we consider the orbit $\C^*[C]$; the curves $C_\mp$, parametrized by the sink $[C_-]$ and the source $[C_+]$ of $\C^*[C]$ will be called the {\em limit curves} of $C$ under the $\C^*$-action.
\end{definition}

\begin{construction}\label{const:Sc}
Given a non-fixed point $[C] \in \cM$ we construct a rational ruled surface $S_C$ as follows: 
we take the closure of the orbit of $[C]$, $\overline{\C^*[C]}$, which is a rational curve, we take its normalization $\nu: \PP^1 \to \overline{\C^*[C]}$ and we consider the fiber product:
\begin{equation}\label{eq:C}
\xymatrix{S_C \ar@/^10pt/[rr]^{\overline q}\ar[d]_p \ar[r]_{\overline \nu}& \cU \ar[d]_p \ar[r]_q& X\\
\PP^1 \ar[r]^\nu &\cM}
\end{equation}
By construction, the  surface $S_C$ is isomorphic to $\mathbb F_e:=\P(\cO_{\P^1}\oplus \cO_{\P^1}(e))$ for some $e\geq 0$, and is endowed with a $\C^*$-action with two invariant fibers $F_-$, $F_+$, which are mapped via $\overline q$ to the limit curves of $C$.\end{construction}

\begin{remark}\label{rem:1contr}
Note that $\overline q(S_C)$ has dimension two, so $\overline q$ can contract at most one curve in $S_C$, otherwise we would have a one-dimensional family of curves of $\cM$ passing by two points of $X$; this, by Bend-and-Break (see \cite[Proposition 3.2]{De}) would contradict the fact that $\cM$ is an unsplit family. If the curve $C$ contains a fixed point of the action, then $S_C$ contains a section which is contracted by $\overline q$.
\end{remark}

\begin{corollary}\label{rem:onefix} Let $C$ be a curve in $\cM$ such that $[C]$ is not a fixed point. Then there is at most one point of $C$ which is fixed. 
\end{corollary}

We will now use Construction \ref{const:Sc} to prove that the equalization of the action on $X$ implies the equalization of the induced $\C^*$-action on $\cM$. 

\begin{proposition}\label{prop:equalM}
The induced $\C^*$-action  on $\cM$ is equalized. In particular, every closure of a $1$-dimensional orbit of the action on $\cM$ is a smooth rational curve.
\end{proposition}

\begin{proof}
Let $[C] \in \cM$ be a non-fixed point, and consider the normalization $\nu: \PP^1 \to \overline{\C^*[C]}$; we will show that the induced action on $\PP^1$ is faithful.
Let $S_C$ be as in Construction \ref{const:Sc}; the $\C^*$-action on $S_C$ is given by a linearization of the action on the vector  bundle $\cE:=\cO_{\P^1}\oplus \cO_{\P^1}(e)$. We will show that there exists a section $C'$ of $S_C \to \PP^1$ that is mapped to the closure of a $1$-dimensional orbit in $X$.  Since the $\C^*$-action on $X$ is faithful, also the action on $C'$ will be faithful, and therefore the action on $\P^1$ will be faithful, too.
We distinguish the cases $e=0$ and $e >0$.

 If $e=0$, we consider the $\C^*$-action on the space of global sections $\HH^0(\P^1,\cE)$, that has two linearly independent eigenvectors, providing two $\C^*$-invariant surjective morphisms $\cE \to \cO_{\P^1}$, that correspond to two $\C^*$-invariant sections of $S_C\to \P^1$, denoted by $C_0,C_1$. We conclude by noting that, by Remark \ref{rem:1contr}, $\overline q$ cannot contract both $C_0$ and $C_1$.
 
If $e>0$, then the induced action on $\HH^0(\P^1,\cE)$ preserves the subspace $\HH^0(\cO_{\P^1}(e))$, which has a $\C^*$-invariant $1$-dimensional complement. Then again we obtain two surjective morphisms $\cE \to \cO_{\P^1}$, $\cE \to \cO_{\P^1}(e)$, corresponding to two $\C^*$-invariant sections  $C_0,C_1$ of $S\to \P^1$, as in the previous case.  
\end{proof}

We finish the section by describing the possibilities for the surface $S_C$ and the limit curves $C_-,C_+$.

\begin{proposition}\label{prop:degC1}
Let $[C] \in \cM$ be a non-fixed point, and let  $S_C$ be as  in Construction \ref{const:Sc}. Then we have the following possibilities for $S_C \simeq \mathbb F_e$ and the limit curves $C_-,C_+$:
\begin{itemize}
\item[0f)] $e=0$, and both $C_-$,$C_+$ are fixed;
\item[0b)] $e=0$, and both $C_-$,$C_+$ are bridges;
\item[1)] $e=1$, one among $C_-$ and $C_+$ is fixed, and the other is a bridge;
\item[2)] $e= 2$, and both $C_-$, $C_+$ are bridges.
\end{itemize}
Moreover, if $C$ contains one fixed point of the action, then only cases 1) and 2) are possible.
\end{proposition}

\begin{proof}
Set $\cE:= \cO_{\P^1}(1)\oplus \cO_{\P^1}(e+1)$ and let $a_-,b_-$ (resp. $a_+,b_+$) be the weights of the action at $\cE_{[C_-]}$ (resp. at $\cE_{[C_+]}$).
By \cite[Lemma 2.16]{WORS1}, up to renaming and normalizing them, we have $a_-=0, a_+=1$ and $b_+=b_-+e+1$. Moreover, since the action on $F_-$ and $F_+$ is equalized, we have that $|b_-| \le 1$ and $0 \le b_+ \le 2$. We also note that the fiber $F_-$ (resp. the fiber $F_+$) is a fixed point component if and only if $b_-=0$ (resp. $b_+=1$).
Therefore we have the following possibilities, which lead to the cases listed in the statement:
 \begin{center}
\begin{tabular}{|c|c|c|c|c|}
\hline
 $b_-$ & $b_+$ & $e$ &$C_-$&$C_+$\\
\hline
$-1$ & $0$ & $0$ & orbit & orbit \\
$-1$ & $1$ & $1$ & orbit & fixed \\
$-1$ & $2$ & $2$ & orbit & orbit \\
\hline
$0$ & $1$ & $0$& fixed & fixed \\
$0$ & $2$ & $1$ & fixed &orbit \\
\hline
$1$ & $2$ & $0$ & orbit & orbit \\
\hline
\end{tabular}
\end{center}

If $C$ contains a fixed point $y$, then every fiber of $S_C$ over $\P^1$ contains a point mapping to $y$. In other words, $S_C$ contains a section over $\P^1$ contracted via $\ol{q}$ to $y$. Since  $C$ is not fixed, 
 we may conclude that $e>0$.  
\end{proof}

\begin{figure}[h!]
\begin{center}
\begin{tabular}{ccc}
\includegraphics[height=1.7cm]{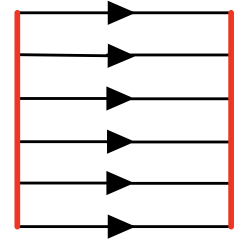} &\includegraphics[height=1.75cm]{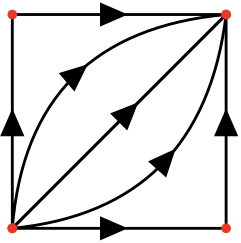} &\includegraphics[height=1.7cm]{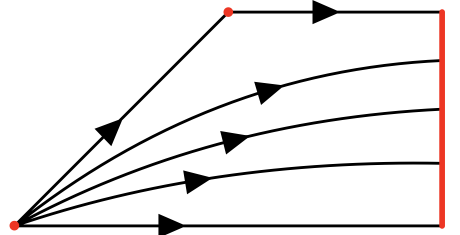}\\
\hline
\quad $e=0$, $C_\pm$ fixed \quad &$e=0$, $C_\pm$ bridges &$e=1, C_+$ fixed \\
\hline
& & \\
\includegraphics[height=1.7cm]{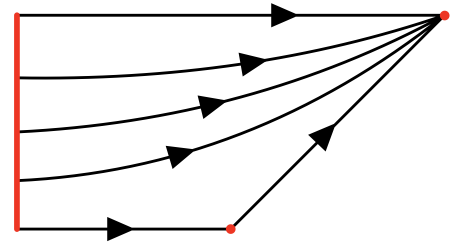} &\includegraphics[height=1.7cm]{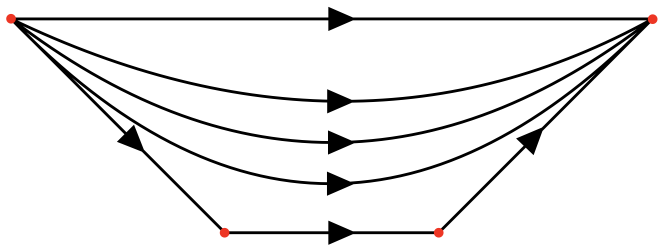} &\\
\hline
 $e=1, C_-$ fixed  &$e=2$, $C_\pm$ bridges & \\
\hline
\end{tabular}
\end{center}
\caption{The possible cases in Proposition \ref{prop:degC1}.}
\end{figure}

\begin{corollary}\label{cor:degC1}
In the situation of Proposition \ref{prop:degC1}, assume moreover that $C$ contains a fixed point $y \in Y$, with $Y\in \cY$. Then:
\begin{itemize}
\item in case 1), either $C_+$ links $Y$ to a component $Y'$ of $L$-weight $\mu_L(Y)+d$, or $C_-$ links $Y$ to a component $Y'$ of $L$-weight $\mu_L(Y)-d$;
\item in case 2), $C_+$, $C_-$ link $Y$ to components  of $L$-weights $\mu_L(Y)+d$, $\mu_L(Y)-d$, respectively.
\end{itemize}
\end{corollary}

\begin{proof}
A general point $x \in C$ is not fixed, hence the sink and source $x_\pm$ of the orbit of $x$ have weights $\mu_L(x_-) < \mu_L(x_+)$. In case 1) if, for instance, $C_-$ is fixed, then $\mu_L(x_-)=\mu_L(y)$ and, by Lemma \ref{lem:AMvsFM}, $\mu_L(x_+)=\mu_L(y)+d$.
In case 2) both $C_-$ and $C_+$ are bridges, then $x_\pm \not \in Y$, and $\mu_L(x_\pm)= \mu_L(Y) \pm d$ by Lemma  \ref{lem:AMvsFM}. Since  $\mu_L(x_-) < \mu_L(x_+)$ we must have $\mu_L(x_-) = \mu_L(Y)-d$,  $\mu_L(x_+)=\mu_L(Y)+d$.
   \end{proof} 

\subsection{Action on $\cU$} \label{ssec:actionU}

We start by showing that the equalization of the $\C^*$-action on $X$ is inherited also by the induced action on $\cU$.

\begin{proposition}\label{prop:equalU}
The induced $\C^*$-action  on $\cU$ is equalized.
\end{proposition}

\begin{proof}
Let $u \in \cU$ be a non-fixed point, and let $G_u,G_{p(u)},G_{q(u)} \subset \C^*$ be the isotropy subgroups of $u\in \cU$, $p(u)\in \cM$, $q(u)\in X$. Clearly we have that $G_u\subset G_{p(u)}\cap G_{q(u)}$. 
Since the closure of $\C^*u$ cannot be contracted by both $p$ and $q$, we have that either $G_{p(u)}$ or $G_{q(u)}$ are finite. Moreover, since the actions of $\C^*$ on $X$ and $\cM$ are equalized, we have that either $G_{p(u)}$ or $G_{q(u)}$ are trivial, and we conclude that $G_u$ is trivial as well. 
 \end{proof}

Since $q$ is equivariant, a fixed point component of the action on $\cU$ must be mapped by $q$ to a fixed point component $Y \in \cY$, hence,
to determine the fixed point components of the action on $\cU$, we will study the $\C^*$-action on the subvarieties 
\[\cU_Y:= q^{-1}(Y), \ \text{with}\ Y \in \cY.\] 
Since $p$ is equivariant, the image of a fixed point in $\cU_Y$ must be a point $[C]$ such that either $C \subset Y$, or $C$ is a bridge. In other words, the fixed point locus of $\cU_Y$ can be written as a disjoint union of (a priori reducible) closed subsets $$\cU_Y^{\C^*}=\cU_Y^-\sqcup \cU_Y^0\sqcup \cU_Y^+$$
defined as:
\begin{itemize}[topsep=10pt,itemsep=2pt]
\item  $\cU_Y^-$: set of points $([C],y)$ such that $C$ is a bridge with source $y \in Y$;
\item $\cU_Y^0$: set of points $([C],y)$ such that $y \in C \subset Y$;
\item  $\cU_Y^+$: set of points $([C],y)$ such that $C$ is a bridge with sink $y \in Y$.
\end{itemize}
Note that $\cU_Y=\bigsqcup_{y\in Y}\cU_y$, and that every $\cU_y$ is $\C^*$-invariant. In particular, every orbit in $\cU_Y$ is contained in $\cU_y$ for some $y \in Y$.
The following Lemma, which is a straightforward consequence of Corollary \ref{cor:degC1}, describes the $\C^*$-orbits on $\cU_Y$:

\begin{lemma}\label{lem:Uorbits}
The orbits of the action on $\cU_Y$ are of three possible types:
\begin{itemize}[topsep=5pt,itemsep=2pt]
\item orbits with sink in $\cU_Y^-$ and source in $\cU_Y^0$;
\item orbits with sink in $\cU_Y^0$ and source in $\cU_Y^+$;
\item orbits with sink in $\cU_Y^-$ and source in $\cU_Y^+$.
\end{itemize}
\end{lemma}

\begin{remark}
Some of the fixed point sets of $\cU_Y$ may be empty, even if the Picard number of $X$ is one; for instance $\cU_Y^- = \emptyset$ if $Y$ is the sink of the action on $X$, and $\cU_Y^+ = \emptyset$ if $Y$ is the source of the action on $X$. 
Also $\cU_Y^0$ may be empty; for instance, we may consider the Lagrangian Grassmannian $\DC_3(3)$ endowed with  the equalized action of  bandwidth three described in \cite[Section 8]{WORS1}. In this example $\cM$ is the family of lines in $\DC_3(3)$, and the inner fixed point components are Veronese surfaces which do not contain lines. 
\end{remark}
 
We will now show that the fixed point components described above are irreducible, with the possible exception of $\cU_Y^0$, which can be reducible only if $\cU_Y^\pm \not = \emptyset.$ 

\begin{lemma}\label{lem:extremal}
If \,$\cU_Y^i$ (for $i=-, 0, +$)  contains an extremal fixed point component, then it is irreducible,  and meets $\cU_y$ for every $y$.
\end{lemma}

\begin{proof}
Assume, without loss of generality, that $\cU_Y^i$ contains the sink of the action on $\cU_Y$. Then, given any point $x \in \cU_Y$ there exists a connected chain of closures of orbits $C_1 \cup \dots \cup C_k$ in $\cU_Y$ such that $C_1$ meets the sink of $\cU_Y$, the source of $C_j$ is the sink of $C_{j+1}$ for every $j$,  and $x \in C_k$.
By Lemma \ref{lem:Uorbits} there are no such chains  joining two different components of $\cU_Y^i$, hence $\cU_Y^i$ is irreducible. 
To prove the last statement it is enough to notice that,  for every $y \in Y$, every connected chain  of closures of orbits meeting $\cU_y$ is contained in $\cU_y$.
 \end{proof}

From Lemmata \ref{lem:Uorbits}  and  \ref{lem:extremal}  we get the following:

\begin{corollary}\label{cor:unif}
If $\cU_Y^- \not= \emptyset$ (resp. $\cU_Y^+ \not= \emptyset$), then it is irreducible, it is the sink (resp. the source) of the action on $\cU_Y$, and meets $\cU_y$ for every $y \in Y$. Moreover, if  $\cU_Y^0$ is reducible, then $\cU_Y^\pm$ are not empty, irreducible, and are the extremal fixed point components of the action.
\end{corollary}

Note that $\cU_Y^0$ may indeed be reducible, as the following example shows.

\begin{example} 
Consider the equalized action on $\P^n$ which has two linear spaces $\P^k$ \and $\PP^{n-k-1}$ as extremal fixed point components $(1 \le k \le n-2)$ and the induced action on the Grassmannian of lines $\DA_n(2)$. 
This action has extremal fixed point components isomorphic to $\DA_k(2)$ 
and $\DA_{n-k-1}(2)$, 
and one inner fixed point $Y$ component isomorphic to $\DA_{k}(1)\times \DA_{n-k-1}(1)$; 
in this case  $\cU_Y^0$, which parametrizes the lines contained in $Y$, has two irreducible components.
\end{example}

An important consequence of the irreducibility of $\cU^\pm$ is the uniqueness of the fixed point components  of the action on $X$ linked to a given one by a bridge.

\begin{corollary}\label{cor:unique}
If $Y \in \cY$ is a fixed point component such that  $\cU_Y^- \not =\emptyset$ (resp.  $\cU_Y^+ \not =\emptyset$) then there exists a unique fixed point component $Y'$ of weight $\mu_L(Y)-d$ (resp. of weight $\mu_L(Y)+d$) which is linked to $Y$ by a bridge. 
\end{corollary}

\begin{proof}
The  assertion follows from the fact that 
\[p(\cU_Y^-) = \bigsqcup_{\mu_L(Y_i)=\mu_L(Y)-d} p(\cU_{Y_i}^+) \cap p(	\cU^-_Y),\] 
and the irreducibility of $\cU_Y^-$. Note that the union is disjoint since, by Lemma \ref{lem:AMvsFM}, there are no curves of $\cM$ meeting two components of the same weight.
 \end{proof}

\section{Action on $X$: orbit graph and BB-cells}\label{sec:fpc}

In this section we will work under the following assumptions:

\begin{setup}\label{set:basicZ}
$X$ is a smooth projective variety of Picard number one, endowed with an equalized $\C^*$-action of criticality $r$, admitting a beautiful family of rational curves $\cM$, with evaluation morphism $q: \cU \to X$ (see Section \ref{ssec:ratcurves}). We denote by $L$ the ample generator of $\Pic(X)$, and by $d$ the $L$-degree of the curves of $\cM$.
\end{setup}

We start by describing the fixed point components of the action on $\cU$, which, as noticed in Section \ref{ssec:actionU}, are contained in $\bigsqcup_{Y \in \cY}\, \cU_Y$, where $\cU_Y=q^{-1}(Y)$. For every $Y \in \cY$ we will consider the fixed point components $\cU_Y^-,\cU_Y^0,\cU_Y^+$, as defined in Section \ref{ssec:actionU}.

\begin{proposition}\label{prop:ufixed} Let $X$ and $\cM$ be as in Setup \ref{set:basicZ}. Let $Y \in \cY$ be a fixed point component of the $\C^*$-action on $X$. 
Then:
\begin{itemize}[leftmargin=25pt]
\item[(e1)] if $Y$ is extremal and $\dim Y=0$, then the action on $\cU_Y$ is trivial;
\item[(e2)] if  $Y$ is the sink (resp. the source) and $\dim Y>0$, then $\cU_Y^0$ and $\cU_Y^+$ (resp. $\cU_Y^-$  and $\cU_Y^0$)  are non-empty, and are the sink and the source of the action on $\cU_Y$. There are no inner fixed point components;
\item[(i)] if $Y$ is an inner fixed point component then $\cU_Y^-$ and $\cU_Y^+$ are non-empty, and are the sink and the source of the action on $\cU_Y$; the inner fixed point components of the action are the irreducible components of $\cU_Y^0$, if $\cU_Y^0 \not = \emptyset$. 
\end{itemize}
\end{proposition}

\begin{proof}
Let us assume first that the fixed point component  $Y\subset X$ is the sink of the action; the case in which $Y$ is the source is analogous.
In case (e1) we have that $\cU_Y^- \sqcup \,\cU_Y^0= \emptyset$ and so, by Lemma \ref{lem:extremal}, $\cU_Y^+$ is irreducible. Hence $\cU_Y$ has a unique fixed point component, and this is only possible if the action on $\cU_Y$ is trivial.

In case (e2), we have that $\cU_Y^-=\emptyset$. We will show that $\cU_Y^0\neq\emptyset$; then, by Lemma \ref{lem:extremal} and  Lemma \ref{lem:Uorbits} we will get that $\cU_Y^0$, $\cU_Y^+$ are, respectively, the sink and the source of the action. 
Since  $\Pic(X)\simeq \Z$, by \cite[Lemma~2.8.1]{WORS1}  the union of closures of the BB-cells $B^-(Y')$  over the fixed point components $Y'$ different from $Y$ has codimension at least two in $X$. In particular, by \cite[Proposition II.3.7]{kollar},  the general curve $C$ of $\cM$ is contained in its complement, which is the BB-cell $X^-(Y)$. It then follows that the limit curve $C_-$ must be contained in $Y$; we conclude that  $\cU_Y^0\neq\emptyset$.

Let now $Y$ be an inner fixed point component. We claim that  $\cU_Y^- \sqcup \cU_Y^+ \not= \emptyset$, i.e., that there exists a bridge with an extremal fixed point in $Y$. 
To prove the claim we notice that, since $X$ is rationally chain connected with respect to the family $\cM$, there exists a curve $C$ of $\cM$ meeting $Y$ and not contained in it. Then, either $C$ is a bridge or, by Corollary \ref{cor:degC1}, one of the limit curves $C_\pm$ is a bridge.

Without loss of generality let us assume that $\cU_Y^- \not = \emptyset$. By Corollary \ref{cor:unique} there exists a unique
fixed point component $Y'$ of weight $\mu_L(Y)-d$ linked to $Y$ by a bridge. 
Let $x$ be a general point in $X^+(Y)\setminus B^+(Y')\neq \emptyset$. By Lemma \ref{lem:trivial}, there exists a curve $C$ of $\cM$ passing through $x$  not contained in $B^+(Y)$, and so $C$ intersects $B^+(Y')\cup B^+(Y)$ in a finite number of points. In other words, given a general point $x'\in C$, we have that $x'_+\notin Y\cup Y'$. Since, on the other hand, $x_+\in Y$, we may conclude, by Lemma \ref{lem:AMvsFM} and Corollary \ref{cor:unique}, that the limit curve $C_+$ is a bridge linking $Y$ with a component $Y''$ of weight $\mu_L(Y'')=\mu_L(Y)+d$. In particular this shows that $\cU_Y^+ \not = \emptyset$. Then $\cU_Y^-$ and $\cU_Y^+$ are the extremal fixed point components by Corollary \ref{cor:unif}.
 \end{proof}

We can now prove the main statement of the Section, which describes the $L$-weights of the action on $X$, shows the existence of a unique fixed point component for each weight, and  the existence of bridges passing by every fixed point.

\begin{theorem}\label{thm:actionx}
Let $X$ and $\cM$ be as in Setup \ref{set:basicZ}. Then:
\begin{enumerate}[leftmargin=25pt]
\item the weights of the action are $0 < d < 2d < \dots < (r-1)d < rd$;
\item for every  $j =0, \dots, r$  there exists a unique  $Y_j \in \cY$ such that $\mu_L(Y_j)=jd$;
\item through every point $y \in Y_j$ with $j \not = 0$ there are bridges linking  $Y_{j-1}$ and $Y_{j}$;
\item through every point $y \in Y_j$ with $j \not=r$ there are bridges linking $Y_j$ and $Y_{j+1}$.
\end{enumerate}
\end{theorem}

\begin{proof}
Let $Y$ be a fixed point component. By Proposition \ref{prop:ufixed}, if $Y$ is the sink (resp. the source), then $\cU_Y^+$ (resp. $\cU_Y^-$) is non-empty, and if $Y$ is an inner fixed point component, then both $\cU_Y^-$, $\cU_Y^+$ are non-empty.

Moreover, by Corollary \ref{cor:unif}, if $\cU_Y^\pm \not = \emptyset$, then $\cU_Y^\pm \cap \cU_y \not = \emptyset$ for every $y \in Y$. In particular, for every fixed point not contained in the sink there exists a bridge with source at that point, and for every fixed point not contained in the source there exists a bridge with sink at that point. Therefore, if $a_i$ is a weight with $0<i<r$ then also $a_i\pm d$ are weights of the action. Since every fixed point component  can be joined to the sink by a connected chain of bridges, by Lemma \ref{lem:AMvsFM} every weight is a multiple of $d$.

It remains to prove the uniqueness in (2) and we proceed by contradiction.  Let $j$ be the minimum index for which there exist two fixed point components $Y_j, Y'_j$ of weight $jd$. Since the sink and the source are  the only fixed point components of weight $0$ and $rd$, respectively, we have that $0 < j <r$. Hence, by the first part of the proof, we have that the the unique component of weight $(j-1)d$  is linked by a bridge to both $Y_j$ and $Y'_j$; this contradicts Corollary \ref{cor:unique}.
 \end{proof}

\begin{notation}
Since for every $L$-weight $a_j=jd$ there exists a unique fixed point component $Y_j$ of weight $a_j$ we will set 
\[
\cU_j:=\cU_{Y_j},\quad \cU_j^{i}:=\cU_{Y_j}^{i} \ (i=+,0,-),\quad B_j^\pm:=B^\pm(Y_j), \quad 
\nu^\pm_j=\nu^\pm(Y_j),
\]
for every $j=0, \dots, r$. 
\end{notation}

\begin{corollary}\label{cor:sections}
For every $j=0,\dots, r-1$ we have $p(\cU_j^+)=p(\cU_{j+1}^-)$.
\end{corollary}

\begin{proof} By Theorem \ref{thm:actionx} every bridge $C$ with source in $Y_j$ has sink in $Y_{j+1}$ and viceversa, therefore every point of $\cU_j^+$ is linked to a unique point of $\cU_{j+1}^-$ by the closure of the orbit $([C],\C^*x)$, with $x \in C$ general, and the statement follows.
 \end{proof}

\subsection{The orbit graph}

Note that, by Theorem \ref{thm:actionx},  we have the following order on the fixed point components of the action on $X$:
\[
Y_0 \orbits Y_1 \orbits Y_2 \orbits \dots \orbits Y_{r-1} \orbits Y_r.
\]
In this section we will show that the orbit graph of the action on $X$ is complete, i.e., that every two fixed point components are linked by the closure of an orbit. 

Recalling that a fixed point of $\cM$ parametrizes either a fixed curve or a bridge, we see that
the fixed point components of the action on $\cM$ are images via $p$ of the varieties $\cU_j^i$, $(i=-, 0, +)$. We then set:
\begin{itemize}[topsep=10pt,itemsep=2pt]
\item $\cM_{j}:=p(\cU_{j}^+)=p(\cU_{j+1}^-)  \quad \text{for} \ j=0, \dots, r-1$;
\item $\cM_j^0:= p(\cU_{j}^0)  \quad \text{for} \ j=0, \dots, r$.
\end{itemize}
By Proposition \ref{prop:ufixed} the fixed point components $\cM_j$ are non-empty and irreducible for every $j=0, \dots, r-1$. The fixed point components $\cM_j^0$ can be empty or reducible, and one may prove the following: 

\begin{lemma}
Every irreducible component of $\cM_j^0$  parametrizes a beautiful family of rational curves in $Y_j$.
\end{lemma}

\begin{proof}
Let $\cN$ be an  irreducible component of $\cM_j^0$. By Remark \ref{rem:beauty}, $\cN$ is an irreducible component of the normalization $\rat^n(Y_j)$ of $\rat(Y_j)$, hence it is a family of rational curves in $Y_j$, and it is obviously unsplit. Moreover, since $T_{Y_j}$ is a direct summand of ${T_X}_{|Y_j}$ (see Section \ref{ssec:c*}), it follows that $\nu^*T_{Y_j}$ is nef, for the normalization $\nu$ of any curve of the family $\cN$, and we conclude that $\cN$ is beautiful by Remark \ref{rem:beauty}.
 \end{proof}

The main result in this section is the following:

\begin{theorem}\label{thm:complete}
For every $i<j$ there exist orbits with sink in $Y_i$ and source in $Y_j$ and orbits with sink in $\cM_i$ and source in $\cM_j$. In particular the orbit graph of the action on $X$ is complete.
\end{theorem}

The proof of this statement will be obtained by combining  Theorem \ref{thm:actionx} together with Lemmata \ref{lem:orbm1}, \ref{lem:orbm2}, \ref{lem:orbm3} below. 

\begin{lemma}\label{lem:orbm1}
For $i=0, \dots, r-2$  there exists an orbit with sink in $\cM_i$ and source in $\cM_{i+1}$.
\end{lemma}

\begin{proof} We consider the action on $\cU_{i}$, which, by Proposition \ref{prop:ufixed}, has sink $\cU_{i}^-$ and source $\cU_{i}^+$. In particular, the image via $p$ of a general orbit is an orbit with sink in $\cM_i$ and source in $\cM_{i+1}$.
 \end{proof}

\begin{lemma}\label{lem:orbm2} For $0 \le i < j \le r-1$,  if there exists an orbit in $\cM$ with sink in $\cM_i$ and source in $\cM_j$, then there exists an orbit in $X$ with sink in $Y_i$ and source in $Y_{j+1}$.
\end{lemma}

\begin{proof}
Denote by $\C^*[C]$ the orbit with sink in $\cM_i$ and source in $\cM_j$.
We consider the rational ruled surface $S_C$ as in Construction \ref{const:Sc}. The limit curves $C_-\in\cM_i$, $C_+\in\cM_j$ are bridges whose sinks and sources we denote, respectively, by $y_i,y_{i+1}\in C_-$,  $y_j,y_{j+1}\in C_+$. 

The action on $S_C$ has four fixed points, which are the sinks and sources of the action on the two $\C^*$-invariant fibers of $p$ over the points $[C_\pm]$.
Since $q$ is equivariant and $Y_i \orbits Y_{i+1} \preceq Y_j \orbits Y_{j+1}$ the sink and the source of the action on $S_C$ are the fixed points which are mapped to $y_i$ and $y_{j+1}$, respectively.

The image via $q$ of a general orbit of the action on $S_C$ is then an orbit in $X$ with sink $y_i\in Y_i$ and source  $y_{j+1}\in Y_{j+1}$.
 \end{proof}

\begin{lemma}\label{lem:orbm3} For $0 \le i +1<j \le r-1$
if  there exists an orbit  with sink  in $Y_{i+1}$ and source in $Y_j$, then there exists an orbit in $\cM$ with sink in $\cM_{i}$ and source in $\cM_{j}$.
\end{lemma}

\begin{proof}
Let $\Gamma$ be the closure of an orbit  with sink in $Y_{i+1}$ and source in $Y_j$, and consider the $\C^*$-invariant subvariety $\cU_\Gamma:=q^{-1}(\Gamma)$. The fixed point components of the action on $\cU_\Gamma$ are the intersections with $\cU_\Gamma$ of the fixed point components 
$\cU^-_{i+1}\orbits \cU_{i+1}^0 \orbits \cU_{i+1}^+\orbits \cU^-_{j}\orbits \cU_{j}^0 \orbits \cU_{j}^+$.

In particular the sink of the action on $\cU_\Gamma$ is contained in $\cU^-_{i+1}=\cU^+_{i}$, and the source is contained in $\cU^+_{j}$.
The image in $\cM$ of a general orbit in $\cU_\Gamma$ is then an orbit  with sink in $\cM_{i}$ and source in $\cM_{j}$.
 \end{proof}

\begin{proof}[Proof of Theorem \ref{thm:complete}]
Let us consider the following two statements, depending on a positive integer $k$:
\begin{itemize}
\item[{[$S_X(k)$]}] There exist an orbit in $X$ with sink in $Y_i$ and source $Y_{i+k}$, for every $i,i+k\in\{0,\dots, r\}$.
\item[{[$S_\cM(k)$]}] There exist an orbit in $\cM$ with sink in $\cM_i$ and source $\cM_{i+k}$, for every $i,i+k\in\{0,\dots, r-1\}$.
\end{itemize}
The statements $S_X(1)$ and $S_\cM(1)$ follow from Theorem \ref{thm:actionx} and Lemma \ref{lem:orbm1}, respectively.
From Lemma \ref{lem:orbm2} we infer that  $S_\cM(k)\Rightarrow S_X(k+1)$,  while Lemma \ref{lem:orbm3} provides  $S_X(k)\Rightarrow S_\cM(k+1)$. Recursively, we get that both statements hold for every possible $k$.
 \end{proof}

\subsection{Geometry of the BB-cells}

In this section we will discuss some special properties of the plus and minus BB-decompositions in the assumptions of Setup \ref{set:basicZ}. We start by showing that in this situation such decompositions are stratifications, i.e., that the closure of a BB-cell is a union of BB-cells:

\begin{proposition}\label{prop:BB1} 
In the situation of Setup \ref{set:basicZ}, we have that:  
\[B_j^- \subset B_i^- \ \text{if}\ i < j \quad \text{and}\quad B_j^+ \subset B_i^+\ \text{if} \ j < i.\] 
In particular, the plus and minus BB-decompositions are stratifications.
\end{proposition}

\begin{proof}
It is clearly enough to prove only the first part of the statement, and to prove it only for $j=i+1$. 
In order to do so, we will show that a general orbit with sink in $Y_{i+1}$ is contained in  $B_i^-$. Note that Theorem \ref{thm:complete} implies that there exists an orbit with sink in $Y_{i+1}$ and source in $Y_r$, from which we may conclude that $B_{i+1}^-$ intersects the open subset $X^+(Y_r)$, i.e., that the general orbit with sink in $Y_{i+1}$ has source in $Y_r$. Let $\Gamma$ be the closure of such an orbit, and set $\cU_\Gamma:=q^{-1}(\Gamma)$.

The sink and the source of the action on $\cU_\Gamma$ are, respectively, $\cU_{i+1}^- \cap \cU_\Gamma$ and, either  $\cU_r^0 \cap \cU_\Gamma$, or $\cU_r^- \cap \cU_\Gamma$ if the source is zero dimensional. Let $\Gamma'$ be a general orbit in $\cU_\Gamma$ and set $\Gamma'':=p(\Gamma')$, which is the closure of an orbit $\C^*[C]$ in $\cM$.

Consider the ruled surface $S_{C}$ as in Construction \ref{const:Sc}. The limit curve $C_-$ is a bridge connecting $Y_i$ and $Y_{i+1}$, while the limit curve $C_+$ is either a curve contained in $Y_r$, or a bridge connecting $Y_{r-1}$ and $Y_{r}$ (if $Y_r$ is an isolated point). In all cases, the sink of the action on $S_C$ is a point mapped in $X$ to a point of $Y_i$, and the source is  mapped to $Y_r$. Hence, a general orbit in $S_C$ is mapped to an orbit contained in $B^-(Y_i)$, and therefore $q(S_C) \subset B^-(Y_i)$; since $\Gamma = q(\Gamma') \subset q(S_C)$, the proof is finished.
 \end{proof}

We will now prove that, for every $i = 0, \dots, r-1$, the intersection $B_i^- \cap B_{i+1}^+$ is proper. These intersections can be regarded as generalizations of the {\em Richardson varieties}, which are defined as intersections of Schubert and opposite Schubert varieties in rational homogeneous varieties. In order to do so we will use the following statement on the existence of $\C^*$-invariant smoothings of $\C^*$-invariant $1$-cycles:

\begin{lemma}\label{lem:deform+bridge} Let $C$ be a bridge with sink $y \in Y_i$ and source $y' \in Y_{i+1}$, and let $C'$ be the closure of an orbit with sink $y'$ and source in $Y_r$. Then the cycle $C+C'$ is the limit of closures of orbits with sink $y$ and source belonging to $Y_r$. 
\end{lemma}

\begin{proof}
Set $\cU':=q^{-1}(C')$. The sink and the source of the action on $\cU'$ are, respectively, $\cU_{i+1}^- \cap \cU'$ and  $\cU_r^0 \cap \cU'$ (or $\cU_r^- \cap \cU'$ if the source is zero dimensional). 
Consider a general orbit in $\cU'$ with sink in $([C],y')$; its image in $\cM$ via $p$ is the closure of an orbit  $\C^*[\Gamma]$.

Consider the ruled surface $S_{\Gamma}$, defined as in Construction \ref{const:Sc}, and the limit curves $\Gamma_\mp$. There exists a ($\C^*$-invariant) section $\Gamma'$ of $p$ such that $q(\Gamma')=C'$,  $q(\Gamma_-)=C$, and $q(\Gamma_+)$ is either a bridge connecting $Y_{r-1}$ and $Y_{r}$, or a curve contained in $Y_r$.  

In the first case $\Gamma'$ connects the source of $\Gamma_-$ with the source of $\Gamma_+$ and we are in case 0b) of Proposition \ref{prop:degC1}; in the second  $\Gamma'$ connects the source of $\Gamma_-$ with  a point  of $\Gamma_+$, which is fixed, and we are in case 1) of Proposition \ref{prop:degC1}.
In both cases the cycle $\Gamma_-+ \Gamma'$ is a limit of closure of general orbits in $S_\Gamma$. These orbits are mapped to orbits with sink $y \in Y_i$ and source belonging to $Y_r$. 
 \end{proof}

\begin{proposition}\label{prop:BB2} For every $i=0, \dots, r-1$ the intersection $B_i^- \cap B_{i+1}^+$ is proper, i.e., $\codim (B_i^- \cap B_{i+1}^+)=\codim B_i^- + \codim B_{i+1}^+ = \nu^+_i + \nu^-_{i+1}$.
\end{proposition}

\begin{proof}
Let us consider again the Chow quotient $\CX$ of $X$ by the $\C^*$-action (cf. \cite[Section 4]{WORS6}), which is an $(n-1)$-dimensional projective variety.

Set $B:=B_i^- \cap B_{i+1}^+$. Since, by Theorem \ref{thm:actionx}, there exists a bridge between $Y_i$ and $Y_{i+1}$, then $B$ intersects the open sets $X^-(Y_i)\subset B_i^-$, $X^+(Y_{i+1})\subset B_{i+1}^+$, and we may claim that the $\C^*$-orbit of a general point in $B$ has its extremal fixed point components in $Y_i,Y_{i+1}$. We then conclude that $B$ is the locus of the set of bridges between $Y_i$ and $Y_{i+1}$, that we denoted by $I$; note that 
$I=\cU_i^+=\cU_{i+1}^-$. In particular  $B$ is irreducible and $\dim B=\dim I+1$.

We denote by $e_{i}:I\to Y_i$ (resp. $e_{i+1}:I\to Y_{i+1}$) the map assigning to every element of $I$ its sink (resp. its source). Let us denote by $V_i^+\subset \P(N^+(Y_i))$, (resp. $V_{i+1}^-\subset \P(N^-(Y_{i+1}))$) the open set parametrizing orbits linking $Y_0$ and $Y_i$ (resp. $Y_{i+1}$ and $Y_r$), which is non-empty by Theorem \ref{thm:complete}. Consider the fiber product:
\[
Z:=\big(V_i^+\times_{Y_i} I\big)\times_I \big(V_{i+1}^-\times_{Y_{i+1}} I\big),
\]
which parametrizes all the connected $\C^*$-invariant $1$-cycles of the form $C' + C + C''$, where $C'$ is the closure of an orbit linking $Y_0$, $Y_i$, $C\in I$ and $C''$ is the closure of an orbit linking $Y_{i+1}$, $Y_r$. The variety $Z$ is smooth, irreducible, of dimension:
\[\dim Z= \dim I+\nu^+_i-1+\nu^-_{i+1}-1,\]
and it admits an injective morphism into $\CX$. 

On the other hand, we may consider in $\CX$ the subvariety parametrizing connected $\C^*$-invariant cycles of the form $C_1+C_2$ with $C_1$ linking $Y_0$ and $Y_i$, $C_2$ linking $Y_{i}$ and $Y_r$. It is isomorphic to an open subset of $\P(N^-(Y_i))\times_{Y_i}\P(N^+(Y_i))$. We denote its closure by $E_i\subset \CX$, and define analogously $E_{i+1}$. By construction $E_i$, $E_{i+1}$ are irreducible divisors in $\CX$. 

The proof will be finished by showing that $Z\subset E_i\cap E_{i+1}$; in fact, since $E_i$, $E_{i+1}$ are obviously different, their intersection in $\CX$ is a union of subvarieties of dimension $n-3$, hence we will have:
$$\dim I+\nu^+_i+\nu^-_{i+1}-2=\dim Z \leq n-3,$$
from which we get
$$\dim B=\dim I+1 \leq n-3-\big(\nu^+_i+\nu^-_{i+1}-2\big)+1=n - \nu^+_i-\nu^-_{i+1}.$$
Then we conclude by noting that the inequality $\codim B\leq\codim B_i^- + \codim B_{i+1}^+$ follows from the smoothness of $X$. 

In order to prove that $Z\subset E_i \cap E_{i+1}$ we have to show that every cycle of $Z$ belongs to $E_i\cap E_{i+1}$. Such cycles are of the form $C'+C+C''\in Z$, where $C',C,C''$ are $\C^*$-invariant irreducible curves linking, respectively, $Y_0$ to $Y_i$, $Y_i$ to $Y_{i+1}$,  and $Y_{i+1}$ to $Y_r$.
 It is then enough to notice that, by Lemma \ref{lem:deform+bridge}, $C+ C''$ deforms algebraically to a $\C^*$-invariant curve linking $Y_0$ to $Y_{i+1}$ and, analogously, that $C'+ C$ deforms algebraically to a $\C^*$-invariant curve linking $Y_i$ to $Y_r$. 
 \end{proof}

An important outcome of the previous statement, that we will use later on, is that it allows us to relate the ranks $\nu^\pm_i$ to the dimensions of the extremal fixed point components of the action on the varieties $\cU_y$, $y\in X^{\C^*}$; more concretely, with the integers:
$$u_i^\pm:= \dim \cU_i^\pm \cap \cU_y,  \mbox{ $y\in Y_i$ general}.$$

\begin{corollary}\label{cor:ump} Let $Y_i \in \cY$ be a fixed point component and $y$ a general point of $Y_i$;  set 
 $\nu^+_{-1}=\nu^-_{r+1}=0$.  Then, for $i=0,\dots, r$, 
\[u_i^+= \nu^-_i-\nu^-_{i+1}-1, \qquad
u_i^-= \nu^+_i-\nu^+_{i-1}-1.\]
In particular, for every $i= 0, \dots, r-1$,
\[ u_i^+ + u_{i+1}^-  = \dim \cU_y.\] 
\end{corollary}

\begin{proof} We start by observing that $B_i^- \cap B_{i+1}^+ = q(p^{-1}(p(\cU_i^+)))$, hence 
\[
\dim  (B_i^- \cap B_{i+1}^+) = \dim \cU_i^+ +1.
\]
Moreover, since $u_i^+$ is the dimension of the general  fiber of $\cU_i^+ \to Y_i$, we can also write 
\[\dim  (B_i^- \cap B_{i+1}^+) = \dim Y_i + u_i^+ -1= n -\nu_i^+-\nu_i^- +u_i^+-1.\]
On the other hand, by Proposition \ref{prop:BB2}, we have $\dim (B_i^- \cap B_{i+1}^+) = n - \nu_i^+ - \nu_{i+1}^-$.
Combining the two equalities we get that $u_i^+= \nu^-_i-\nu^-_{i+1}-1$; a similar argument provides  $u_i^-= \nu^+_i-\nu^+_{i-1}-1$.
Combining these two formulae with (\ref{eq:can}, \ref{eq:dimMx}), we obtain the last stated equality.
 \end{proof}

\section{Actions with an isolated extremal fixed point}\label{sec:isolated}

The goal of this section will be the proof of Theorem \ref{thm:CP1}. We will work in the following setup:

\begin{setup}\label{setup:isolated} $(X,L)$ is a polarized pair, where $X$ is a Fano manifold with $\Pic(X)=\Z L$, 
admitting a beautiful family of rational curves $\cM$, with evaluation morphism $q: \cU \to X$ of relative dimension $m$.
On $X$ there is an equalized $\C^*$-action with an isolated extremal fixed point, which,
without loss of generality, we will assume to be the sink;
we will denote by $r$ the criticality of the $\C^*$-action on $X$, and by $Y_0=\{y_0\},Y_1,\dots, Y_r$ the fixed point components.  
\end{setup}

In the next section we will present some consequences of the assumption that an extremal fixed point of the action is isolated.  
We will then divide the proof in two cases, according to the possible values of the relative Picard number $\rho(\cU/X)$.

\subsection{Proof of Theorem \ref{thm:CP1}: toolbox}\label{ssec:relpic} 

We start by noting that, since the action has isolated sink, we may compute the degree of the ample generator of $\Pic(X)$ on the curves of the family $\cM$: 

\begin{lemma}\label{lem:deg1}
In the situation of Setup \ref{setup:isolated}, the $L$-degree of the curves of $\cM$ is equal to one.
\end{lemma}

\begin{proof}
Let $B^-(Y_1)$ be the closure of $X^-(Y_1)$; the fact that $\nu^+(Y_1)=1$ (see \cite[Lemma~4.5]{WORS3}), implies that $B^-(Y_1)\subset X$ is a divisor. This divisor is locally isomorphic to $N^-(Y_1)$ in a neighbourhood of $Y_1$; in particular it is smooth along $Y_1$. On the other hand, a bridge $C$ of the family $\cM$ linking $Y_1$ with $y_0$ is smooth at its source $y_1$ (Lemma \ref{lem:AMvsFM}), and meets $B^-(Y_1)$ transversally at $y_1$. This implies that $B^-(Y_1) \cdot C =1$, and so $L\cdot C=1$, since $L$ is the ample generator of $\Pic(X)$.
 \end{proof}

\begin{corollary}
The weights of the $\C^*$-action on $X$ are $0,1, \dots, r$; in particular the action has bandwidth $r$ with respect to $L$.
\end{corollary}

\begin{proof} It follows from Lemma \ref{lem:deg1} and Part (1) of Theorem \ref{thm:actionx}.
 \end{proof}

Since the $L$-degree of curves of $\cM$ is one, then $$\cE:=p_*q^*L$$ is a rank two vector bundle, satisfying that $$\cU\simeq\P_\cM(\cE),\qquad \cO_{\P_\cM(\cE)}(1)\simeq q^*L.$$ By the ampleness of $L$ we have that $\cE$ is nef, and so is the line bundle $\lb \in \Pic(\cU)$ defined as 
\begin{equation}\label{eq:D}
\lb:=p^*\det(\cE).
\end{equation}

In the next statement we will identify the variety $\cU_{y_0}$ with a fixed point component of $X$. In fact, we will show that the curves of the family $\cM_{y_0}$ are uniquely determined by their sources, which are the points of $Y_1$; we also know that the curves of $\cM_{y_0}$ are smooth at $y_0$, then the universal family over $\cM_{y_0}$ contains two sections: one is $\cU_{y_0}$ and the other is isomorphic to $Y_1$.

\begin{lemma}\label{lem:det2}
In the situation of Setup \ref{setup:isolated}, we have isomorphisms:
$$
\cU_{y_0}\simeq Y_1,\qquad p^*\cE_{|\cU_{y_0}}\simeq\cO_{\cU_{y_0}}\oplus q^*L_{|\cU_{y_0}}.
$$
In particular, we have that $L_{|Y_1}\simeq D_{|\cU_{y_0}}$.
\end{lemma}

\begin{proof}

By Corollary \ref{cor:sections},  $p(\cU^+_{0})=p(\cU^-_1)\subset\cM$. Moreover, the subvarieties $\cU^+_{0},\cU^-_{1}$, of the $\P^1$-bundle $p^{-1}(p(\cU^+_0))\to p(\cU^+_0)$ are sections (they correspond, respectively, to the sinks and the sources of the curves parametrized by $p(\cU^+_0)$), therefore they are isomorphic. 
Since $\nu^+(Y_1)=1$ (see \cite[Lemma~4.5]{WORS3}) we have that $\cU^-_1\simeq Y_1$. On the other hand, by Proposition \ref{prop:ufixed}, part (e1),  the action on $\cU_{0}= \cU_{y_0}$ is trivial, so we conclude that  $\cU_{0}=\cU_{0}^+\simeq \cU_1^-\simeq Y_1$.

Set $\cE_0:=p^*\cE_{|\cU_{y_0}}$. The image $\cM_{y_0}$ of  $\cU_{y_0}$ into $\cM$ via $p$ is the subfamily of $\C^*$-invariant rational curves with sink $y_0$ and source in $Y_1$, and $\P(\cE_0)=\cU\times_{\cM}\cU_{y_0}$ is the universal family of these curves. 

There are two disjoint sections, $s_0:\cU_{y_0}\to \cU_0^+\subset \P(\cE_0)$, $s_1:\cU_{y_0}\to \cU_1^-\subset \P(\cE_0)$, corresponding to the extremal fixed points of the curves, which map via $q$ to the point $y_0$ and to $Y_1$, respectively. In particular, $\cE_0$ splits as 
$$\cE_0\simeq s_0^*q^*L_{|\cU_{y_0}}\oplus s_1^*q^*L_{|\cU_{y_0}}.$$
Since, via the map $q$, $\cU_0^+$ is contracted to $y_0$ and $\cU_1^-$ maps isomorphically to $Y_1$, it follows that $\cE_0\simeq\cO_{\cU_{y_0}}\oplus q^*L_{|\cU_{y_0}}$, from which we finally get that $\det(\cE_0)\simeq q^*L_{|\cU_{y_0}}\simeq L_{|Y_1}$.
 \end{proof}

As a consequence, we may state the following:

\begin{corollary}\label{cor:det3}
Let $C$ be a rational curve contained in $Y_1$, satisfying that $L\cdot C=1$. Then $C$ is a curve of the family $\cM$.
\end{corollary}

\begin{proof}
Let $f:\P^1\to C$ be the normalization of $C$. We consider the family of closures of orbits linking the curve $C$ with $y_0$, which, as in Lemma \ref{lem:det2}, has universal family $\P(f^*(p^*\cE)_{|\cU_{y_0}}))\simeq \P(\cO_{\P^1}\oplus \cO_{\P^1}(1))$. Let us denote by $q':\P(\cO_{\P^1}\oplus \cO_{\P^1}(1))\to X$ its evaluation map. It contracts the section corresponding to $\cO_{\P^1}\oplus \cO_{\P^1}(1)\to \cO_{\P^1}$ to the point $y_0$, and sends the section corresponding to $\cO_{\P^1}\oplus \cO_{\P^1}(1)\to \cO_{\P^1}(1)$ onto the curve $C$. In particular the Stein factorization of $q'$ is the contraction of $\P(\cO_{\P^1}\oplus \cO_{\P^1}(1))$ to $\P^2$. Then $C$ is algebraically equivalent to the image of a line in $\P^2$ passing by $y_0$, that is to an element of $\cM$. Since $\cM$ is a smooth irreducible component of $\rat^n(X)$, then $C$ is a curve of the family $\cM$, as well.
 \end{proof}

\begin{lemma}\label{lem:det}
The line bundle $\lb$ has positive degree on rational curves contracted by $q$. Moreover, if $y \in X^{\C^*}\setminus (Y_0\cup Y_r)$ is an inner fixed point, then the degree of $\lb$ on the closure of a general orbit of the induced action on $\cU_y$ is two.
\end{lemma}

\begin{proof}
Let $\Gamma$ be a rational curve contracted by $q$,  denote by $\nu:\PP^1 \to p(\Gamma)$ the normalization of $p(\Gamma)$ and consider the Cartesian square
\begin{equation}
\xymatrix{\P(\nu^*\cE) \ar@/^10pt/[rr]^{\overline q}\ar[d]_p \ar[r]_{\overline \nu}& \cU \ar[d]_p \ar[r]_q& X\\
\PP^1 \ar[r]^\nu &\cM}
\end{equation}

By construction the bundle $\nu^*\cE$ is nef, and the surface $\P(\nu^*\cE)$ contains the normalization $\tl{\Gamma}$ of $\Gamma$, which is contracted to a point in $X$. In particular, $\tl{\Gamma}$ is a section of $\P(\nu^*\cE)$,  $\cO_{\P(\nu^*\cE)}(1)\cdot \tl{\Gamma}=0$, and 
$\nu^*\cE$ is of the form $\cO_{\P^1}\oplus \cO_{\P^1}(e)$ for some $e\geq 0$. In the case $e=0$ we would have a curve in $\cM$ whose points parametrize the same curve in $X$, a contradiction. It follows that  $\lb \cdot \tl{\Gamma} = e > 0$. 

For the last claim, we notice that, if $y \in X^{\C^*}\setminus (Y_0\cup Y_r)$ and
 $\Gamma$ is the closure of a general orbit of the induced action on $\cU_y$,  then $e=2$ by Proposition \ref{prop:degC1}.
 \end{proof}

\begin{corollary}\label{cor:unspfam} Let $\Gamma$ be a rational curve in a fiber of $q$ such that $D\cdot \Gamma=1$. Let $V$ be the  component of $\rat^n(\cU)$ containing the point parametrizing $\Gamma$. Then $V$ is an unsplit family. 
\end{corollary}

\begin{proof}
Assume that $\Gamma$ degenerates as $\Gamma_1 + \Gamma_2$. Then, from $
(D+q^*L)\cdot \Gamma=1$ we get that, for some $i=1,2$, we have $(D+q^*L)\cdot \Gamma_i \le 0$, which is impossible, since $q^*L$ is nef and $D$ has positive degree on rational curves.
 \end{proof}

We can now use the above results to bound the Picard number of  $\cU_y$ for every point  $y$ in $Y_1$.

\begin{lemma}\label{lem:bw2} 
If $y$ is a point  of $Y_1$ then the action on $\cU_y$ has isolated sink. In particular, $\cU_y$ is rationally connected. Moreover $\rho(\cU_y) \le 2$.
\end{lemma}

\begin{proof}  By \cite[Lemma 4.5]{WORS3} we have $\nu^+(Y_1)=1$, hence the sink $\cU_y \cap \cU_1^-$ (which is non-empty by Corollary \ref{cor:unif}) is zero dimensional.  Every point of $\cU_y$ can be connected to the sink  by a chain of closures of orbits, so $\cU_y$ is rationally chain connected, hence rationally connected.

For the second part, note that the restriction to $\cU_y$ of the line bundle $D=p^*(\det(p_*q^*L))$ satisfies, by Lemma \ref{lem:det}, the hypotheses of Theorem \ref{thm:CP1rho2}, which allows us to conclude that $\rho(\cU_y) \le 2$.  
\end{proof}

\begin{corollary}\label{cor:picF}
For every $x\in X$ the variety $\cU_x$ is a smooth rationally connected variety. Moreover, denoting by $\Nu(\cU_x,\cU)$ the image of $i_*:\Nu(\cU_x) \to \Nu(\cU)$, we have:
\[ \rho(\cU_x)= \dim \Nu(\cU_x,\cU).\]
In particular, if $\rho(\cU/X)=1$, then $\rho(\cU_x)=1$ for every $x \in X$.
\end{corollary}

\begin{proof}
By Lemma \ref{lem:bw2}, the fiber $\cU_y$ of $q$ over a  point of $Y_1$ is rationally connected. Since rational connectedness is an open and closed property on smooth proper families, all the fibers of $q$ are rationally connected. Since $X$ is rationally connected, by \cite[Corollary~1.3]{GHS} also $\cU$ is rationally connected. 
Now we follow the arguments of \cite[Proposition 3.1]{Kan5}: for every $x \in X$ the map 
\[
\Pic(\cU) \otimes \QQ \simeq H^2(\cU,\QQ) \to H^2(\cU_x,\QQ) \simeq \Pic(\cU_x) \otimes \QQ
\]
is surjective by Deligne's Invariant Cycle Theorem. In particular, for every
$x \in X$ we have $\rho(\cU_x)= \dim N_1(\cU_x,\cU)$.
If $\rho(\cU/X)=1$, then $q$ is a Mori contraction of fiber type, hence $\dim N_1(\cU_x,\cU) \le \rho(\cU/X)$, and so $\rho(\cU_x)=1$.
 \end{proof}

\subsection{Proof of Theorem \ref{thm:CP1}: sketch.}\label{sec:sketch}

Let us start by noting that Theorem \ref{thm:CP1} holds, by Proposition \ref{prop:FHnonhomo}, 
in the case in which the bandwidth of the action on $(X,L)$ is equal to two. 
Then we may assume, without loss of generality, that the bandwidth of the action is at least $3$. In this case the proof will follow from the description of the varieties $\cU_x$, for every $x\in X$. We will show that every $\cU_x$ is isomorphic to $\cU_y$, for $y\in Y_1$, and that $\cU_y$   is either one of the rational homogeneous cases of Theorem \ref{thm:rho1} if $\rho(\cU/X)=1$, or  a product of two projective spaces if $\rho(\cU/X)\ge 2$. 

In both cases we know that the variety $\cU_x$ coincides at every point with the family of tangent directions to lines at a point of an irreducible Hermitian symmetric space (see Table \ref{tab:hermite3}). Hence, the statement will follow from Theorem \ref{thm:OSWi}. 

\subsection{Case $\rho(\cU/X)=1$.}\label{ssec:rho1}

By Corollary \ref{cor:picF}, for every $x \in X$ the variety $\cU_x$ is rationally connected of Picard number one, hence a Fano manifold. In particular the divisor $\lb=p^*(\det(p_*q^*L))$, which has positive degree on rational curves in $\cU_x$ by Lemma \ref{lem:det}, must be ample on $\cU_x$. 

Now we consider a point $y\in Y_1$, and claim that the variety $\cU_y$ is a rational homogeneous manifold. As observed in Lemma \ref{lem:bw2} the $\C^*$-action on $(\cU_y,D_{|\cU_y})$ has isolated sink and, by Lemma \ref{lem:det}, it has bandwidth two. 

If the criticality of the action is one (i.e. if the action on $\cU_y$ has no inner fixed point components) then $\cU_y$ is a projective space by Proposition \ref{prop:crit1}. More precisely, the polarized pair $(\cU_y,D_{|\cU_y})$ is isomorphic to $(\P^m,\cO_{\P^m}(2))$.

If the criticality is two, then  $\cU_y$ is one of the varieties appearing in Theorem \ref{thm:rho1}, so we only need to exclude the non-homogeneous cases. 

The morphism $q:\cU \to X$, which is smooth of relative dimension $m$, is the contraction of an extremal ray, generated by the class of a curve of degree one with respect to $\lb$; denote by $k$ the anticanonical degree of this class. By Lemma \ref{lem:det}, if $y$ is an inner fixed point of the action on $X$,  the anticanonical degree of a general orbit of the induced action on $\cU_y$ is $2k$. 
Then, by Corollary \ref{cor:can}, for every inner fixed point component we have:
\[
2k= 2m- (u_i^-+u_i^+).
\]
Summing on the inner fixed point components and using Corollary \ref{cor:ump}
we get 
\[
2k(r-1) = m(r-2)+ u_0^++u_r^-= m(r-1) +u_r^-;
\]
the last equality follows from the fact that the sink is isolated. 
In particular $2k > m$. Then, recalling that we are assuming $r \ge 3$,
\begin{equation}
2 (2k-m) \le (2k-m)(r-1)= u_r^- \le m.
\end{equation}
Let us note that, if equality holds, then $r=3$ and $u_r^- = m$; in particular the action on $\cU_r$ is trivial, hence, by Proposition \ref{prop:ufixed}, $\dim Y_r=0$. This case has been treated in \cite[Theorem 8.10]{WORS1}.
Moreover, if $u_r^- = m-1$, then,  for $x \in Y_r$, the action on $\cU_x$, which has criticality one, has divisorial sink; since $\rho(\cU_x)=1$ the source of the action on $\cU_x$ must be isolated, and $\cU_x$ is a projective space by Proposition \ref{prop:crit1}, a contradiction.
We can then assume $2 (2k-m) < m-1$, that is, $4k \le 3m-2$.
Listing the varieties in Theorem \ref{thm:rho1} (5) and (6), with the corresponding values of $m$ and $k$, we see that this inequality is never satisfied.\par\medskip 
\begin{center}\renewcommand{\arraystretch}{1.1}
\begin{tabular}{|C|C|C|}
\hline
\text{Variety}& m& k \\\hline\hline
\DD_5(5)
\cap H_1&9&7\\\hline
\DD_5(5)
\cap H_1 \cap H_2&8&6 \\\hline
\DD_5(5)
\cap H_1 \cap H_2\cap H_3&7&5\\\hline
\DA_4(2)
\cap H&5&4\\\hline
\DA_4(2)
\cap H_1 \cap H_2&4&3\\\hline
\end{tabular}
\end{center}\par\medskip

We then conclude that $\cU_y$ is a rational homogeneous manifold. Since rational homogeneous manifolds have no non-trivial infinitesimal deformations, it follows that  $\cU_{x}\simeq \cU_y$ for a  general $x\in X$.  Since every possible case for $\cU_y$ that we have is rigid under smooth Fano deformations (cf. \cite[Main Theorem]{HMdef}), it follows that every $\cU_x$ is isomorphic to $\cU_y$, and we may conclude the proof by \cite[Theorem~1.1]{OSWi}.

\subsection{Case $\rho(\cU/X)=2$.}\label{ssec:rho2}

We consider again a point $y\in Y_1$. As observed in Lemma \ref{lem:bw2} the $\C^*$-action on $(\cU_y,D_{|\cU_y})$ has isolated sink and, by Lemma \ref{lem:det}, it has bandwidth two. Then, by Theorem \ref{thm:CP1rho2}, either  
\begin{enumerate}
\item[(P)]  there exists $0<a<m$ such that $\cU_y \simeq \PP^a \times \PP^{m-a}$, or  
\item[(B)]  there exists $0 < a < m$ such that $\cU_y$ is isomorphic to the blowup of $\P^m$ along a linear space of dimension $a-1$.
\end{enumerate}

We will now construct two unsplit families of rational curves on $\cU$, starting from rational curves in $\cU_y$. We denote by $\ell_1$ and $\ell_2$ minimal rational curves in the extremal rays of $\cU_y$. In case (P) they are lines in the factors of $\cU_y$; in case (B) we choose $\ell_1$ in the ray corresponding to the fiber-type contraction $\cU_y \to \PP^{m-a}$ and $\ell_2$ in the ray contracted by the blowup map. Note that, by Corollary \ref{cor:picF}, the classes of the curves $\ell_1$ and $\ell_2$ are numerically independent in $\Nu(\cU)$.

Denote by $V^1$ and $V^2$ the components  of $\rat^n(\cU)$ containing the elements corresponding to $\ell_1$ and $\ell_2$. By Lemma \ref{cor:unspfam}, $V^1$ and $V^2$ are unsplit families. Moreover, since $\ell_1$ is a free rational curve in $\cU_y$ and the normal bundle of $\cU_y$ in $\cU$ is trivial, then $\ell_1$ is a free rational curve in $\cU$; in particular the family $V^1$ is a covering family.

In case (P) a similar argument shows that also $V^2$ is an unsplit covering family.
Using that, by adjunction, $-K_{\cU}|_{\cU_y}=-K_{\cU_y}$, we can compute the anticanonical degree of curves parametrized by $V^1$ and $V^2$; those degrees are, respectively,  $a+1$ and  $m-a+1$. Let $\cU_x$ be any fiber of $q$. Then the curves parametrized by $V^i$ meeting $\cU_x$ are contained in $\cU_x$, so $\cU_x$ admits two unsplit covering families of rational curves such that the sum of their anticanonical degrees is $\dim \cU_x +2$. We conclude that $\cU_x \simeq \P^a \times \P^{m-a}$ by  \cite[Main Theorem]{O}, for every $x\in X$, and 
this implies that $X$ is a rational homogeneous variety  by \cite[Theorem~1.1]{OSWi}.

We will finish the proof by showing that case (B) cannot occur.
We claim first that, in this case,  $q(\loc(V^2))=X$. To prove the claim, since $V^2$ is unsplit, it is enough to check that $q|_{\loc(V^2)}$ has surjective differential at a point of $\ell_2$; we claim first that this is equivalent to saying that the evaluation of global sections
\begin{equation}
H^0({T_\cU}_{|\ell_2})\otimes\cO_{\ell_2} \to {q^*T_X}_{|\ell_2} \label{eq:evaluation}
\end{equation}
is surjective. 
In fact, as a component of $\rat^n(\cU)$, $V^2$ is the normalization of the image of a subscheme of $\Hom(\P^1,\cU)$, that we denote by $\cH_2$. Our claim follows from proving that the composition $\P^1\times \cH_2\to \cU \to X$ has surjective differential at a point $(P,\ell_2)$, for some $P\in\ell_2$ or, equivalently, that the composition 
\[\P^1\times \cH_2\subset \P^1\times \Hom(\P^1,\cU)\to \Hom(\P^1,X) \to X\] 
has surjective differential at $(P,\ell_2)$; this follows from \cite[II~Proposition~3.4]{kollar}.

Let us now prove the surjectivity of the map in (\ref{eq:evaluation}).
Since \[{T_{\cU_y}}_{|\ell_2} \simeq \cO_{\ell_2}(2) \oplus \cO_{\ell_2}(1)^{ m-a-1} \oplus \cO_{\ell_2}^{a-1} \oplus \cO_{\ell_2}(-1),\] and $q^*{T_X}_{|\ell_2}$ is trivial, then
we have a commutative diagram with exact rows of the form:

$$
\xymatrix{
0\ar[r]&\HH^0({T_{\cU_y}}_{|\ell_2})\otimes\cO_{\ell_2}\ar[r]\ar[d]&\HH^0({T_{\cU}}_{|\ell_2})\otimes\cO_{\ell_2}\ar[r]\ar[d]&\HH^0({q^*T_X}_{|\ell_2})\otimes\cO_{\ell_2}\ar[r]\ar[d]^{\simeq}&0\\
0\ar[r]&{T_{\cU_y}}_{|\ell_2}\ar[r]&{T_{\cU}}_{|\ell_2}\ar[r]&{q^*T_X}_{|\ell_2}\ar[r]&0
}
$$
from which the claimed surjectivity follows.

It follows that, for every $x \in X$, the variety $\cU_x$ admits two unsplit families of rational curves, $V_x^1$ and $V_x^2$, of degree one with respect to $D$ and such that $V_x^1$ is a covering family. In particular, this is the case for $x=y_0$. Then, by Lemma \ref{lem:det2}, we see that $Y_1$ admits two unsplit families of rational curves of $L$-degree one, $W^1$ and $W^2$, such that $W^1$ is a covering family. By Corollary \ref{cor:det3}, curves parametrized by these two families belong to $\cM$. In particular, for a point $y \in Y_1$ belonging to $\loc(W^2)$, $\cU_y^0=\cU_{Y_1}^0\cap \cU_{y}$ has two disjoint components. This contradicts the fact that, by assumption, $\cU_y$ is $\C^*$-equivariantly isomorphic to the blowup of a projective space endowed with a $\C^*$-action that has a unique inner fixed point component (see the Proof of Theorem \ref{thm:CP1rho2}).


\end{document}